\documentclass[preprint]{elsarticle}

\usepackage{lineno,hyperref}
\modulolinenumbers[5]

\usepackage{graphicx}
\usepackage{amssymb}
\usepackage{amsmath}
\usepackage{algorithmicx}
\usepackage[ruled]{algorithm}
\usepackage{algpseudocode}

\date{June 19, 2020}

\ifpdf
\DeclareGraphicsExtensions{.eps,.pdf,.png,.jpg}
\else
\DeclareGraphicsExtensions{.eps}
\fi

\def\ba{\mathbf{a}}

\def\bd{\mathbf{d}}

\def\bg{\mathbf{g}}

\def\bs{\mathbf{s}}

\def\bx{\mathbf{x}}
\def\by{\mathbf{y}}

\newcommand{\N}{\mathbb{N}} 
\renewcommand{\Re}{\mathbb{R}}

\renewcommand{\tilde}{\widetilde}
\renewcommand{\bar}{\overline} 

\newtheorem{theorem}{Theorem}

\newtheorem{remark}{Remark}
\newenvironment{proof}{\vspace*{-5mm}\paragraph{\textnormal{\emph{Proof}}}}{\hfill$\square$ \par\medskip\smallskip}

\begin{document}
    
    \begin{frontmatter}
        
        \title{Using gradient directions to get global convergence of Newton-type methods}
        
        \author[1]{Daniela di Serafino\corref{cor1}}\ead{daniela.diserafino@unicampania.it}
        \author[2]{Gerardo Toraldo}\ead{toraldo@unina.it}
        \author[1]{Marco Viola}\ead{marco.viola@unicampania.it}
        
        \address[1]{Dipartimento di Matematica e Fisica, Universit\`a degli Studi della Campania \\ ``Luigi Vanvitelli'', viale A.~Lincoln~5, 81100 Caserta, Italy}
        \address[2]{Dipartimento di Matematica e Applicazioni, Universit\`a degli Studi di Napoli Federico II, Cupa Nuova Cintia~21, 80126 Napoli, Italy}
        
        \begin{abstract}
            The renewed interest in Steepest Descent (SD) methods following the work of
            Barzilai and Borwein \cite{barzilai_borwein_88} has driven us to consider a globalization strategy based on SD,
            which is applicable to any line-search method.
            In particular, we combine Newton-type directions
            with scaled SD steps to have suitable descent directions.
            Scaling the SD directions with a suitable step length makes
            a significant difference with respect to similar globalization approaches,
            in terms of both theoretical features and computational behavior. We apply our strategy to Newton's method
            and the BFGS method, with computational results that appear interesting
            compared with the results of well-established globalization strategies devised ad hoc for those methods. \\
            
            \noindent
            \emph{AMS subject classification:}  65K05, 90C30, 49M15.
            \begin{keyword}
                Newton-type methods, globalization strategies, steepest descent step. \\
            \end{keyword}
        \end{abstract}
        
        
    \end{frontmatter}
    \section{Introduction\label{sec:intro}}
    We are concerned with the following optimization problem:
    \vspace{-2pt}
    \begin{equation}\label{the problem}
    \mathrm{minimize} \, f(\bx), \;\; \bx \in \Re^n,
    \end{equation}
    where $f : \Re^n \to \Re$ is twice continuously differentiable. Hereafter $\bg(\bx)$ and $H(\bx)$
    denote the gradient and the Hessian of $f$, respectively.
    
    A strictly monotone line-search method for solving \eqref{the problem} generates a sequence $\{ \bx_{k} \}$ as follows:
    \vspace{-3pt}
    \begin{equation}
    \bx_{k+1}=\bx_{k}+\alpha_k\bd_k, 
    \label{linesearch_iteration}
    \end{equation}
    where $\bd_k$ is a descent direction, $\alpha_k>0$ is a step length and $f(\bx_{k+1}) < f(\bx_{k})$.
    For simplicity of notation, we define $f_k = f(\bx_k)$, $\bg_k = \bg(\bx_k)$ and $H_k = H(\bx_k)$.
    The direction $\bd_k$ must be of strict descent, i.e.,
    \begin{equation}
    \label{descent_condition}
    \bg_k^\top \bd_k <0.
    \end{equation}
    However, condition \eqref{descent_condition} alone does not ensure
    convergence, and $\bd_k$ must satisfy, e.g., the angle criterion
    \begin{equation}
    \label{angle_criterion}
    \cos \langle -\bg_k, \bd_k \rangle = \frac{- \bg_k^\top  \bd_k}{\|\bg_k\|\,\|\bd_k\|} \geq \varepsilon_k, \quad \varepsilon_k > 0,
    \end{equation}
    where the sequence $\{\varepsilon_k\}$ is bounded away from 0, which means that the angle between the search direction
    and the Steepest Descent (SD) direction must be bounded away from the right angle.
    
    The step length must usually satisfy the Armijo condition
    \begin{equation}
    f(\bx_{k}+\alpha_k\bd_k)-f(\bx_{k}) \leq
    \sigma_1\alpha_k \bg_k^\top \bd_k, \quad \sigma_1 \in (0,1),\label{armijo} 
    \end{equation}
    or the Wolfe conditions, i.e., \eqref{armijo} and
    \begin{equation}
    \nabla f(\bx_{k}+\alpha_k\bd_k)^\top \bd_k
    \geq \sigma_2\alpha_k \bg_k^\top \bd_k, \quad \sigma_2 \in (\sigma_1,1).\label{wolfe}
    \end{equation}  
    We note that \eqref{armijo} forces a \textit{sufficient decrease} in the objective function, while the \textit{curvature} condition~\eqref{wolfe} prevents 
    the method from taking too small steps, which is not guaranteed by condition~\eqref{armijo} alone.
    This drawback can be avoided by choosing $\alpha_k$ with a suitable backtracking procedure~\cite[page~37]{nocedal_06}.
    
    In a Newton-type (NT) method,
    the search direction $\bd_k^{\mathrm{NT}}$ is computed as the solution of the linear system
    \begin{equation}
    S_k\bd =-\bg_k
    \label{scaled_gradient}
    \end{equation}
    where $S_k$ is some symmetric matrix, possibly positive definite so that~\eqref{descent_condition} automatically holds.
    With the choice $S_k=I$, where $I$ is the identity matrix, the NT method reduces to the classical SD method,
    which is globally convergent with at most linear rate.
    Recently, several attempts have been made to get more efficient SD methods.
    In particular, starting from the seminal work by Barzilai and Borwein (BB)~\cite{barzilai_borwein_88},
    it has been observed that appropriate choices of $\alpha_k$ can, to some extent,
    remedy the slow convergence of the SD method, even for the solution of constrained problems by gradient projection strategies.
    This led to effective algorithms~\cite{crisci_2019,diserafino_2018,diserafino_2018siopt,porta_2015,pospisil_2018}, which have been
    successfully used in several applications~\cite{antonelli_2016,deasmundis_2016,diserafino_2020,dostal_18,zanella_2009}.
    
    With the inclusion of second-order information through $S_k$ we expect a better rate of convergence. However, the search direction
    $\bd_k^{\mathrm{NT}}$ does not guarantee global convergence, even if $S_k$ is positive definite for each $k$. 
    An example is provided by the classical Newton's method, where $S_k= H(\bx_k)$. If $\bx^*$ solves~\eqref{the problem}, with $H(\bx^*)$ positive definite, and $H(\bx)$ is Lipschitz continuous around $\bx^*$, Newton's method
    has local, but not global, convergence with quadratic rate~\cite{fletcher_2000}. However, a suitable reduction of the Newton step
    allows global convergence in the convex case (see, e.g., \cite[page~34]{nesterov_04}).
    In the nonconvex case the Newton direction 
    may not be a descent direction.  Therefore, modifications of Newton's method
    have been developed that replace $H(\bx_k)$ by $\tilde{H}_k=H(\bx_k) + E_k$, where $E_k$ is a symmetric matrix such that
    $ \tilde{H}_k $ is positive definite~\cite{more_84} and the solution $\bd_k^{\mathrm{MN}}$ of $\tilde H_k \bd=-\bg_k$ is a
    descent direction at $\bx_k$. We will refer to these methods as Modified Newton's (MN) methods.
    This approach can be extended to the general framework of Newton-type methods, in which, given an approximation $S_k$ of $H(\bx_k)$,
    one may consider a matrix $E_k$ such that $\tilde{S}_k = S_k+E_k$ is ``sufficiently positive definite'' and $\| E_k \|$ is not much larger
    than $\inf\{\|E\| \, : \, S_k +E \succ 0 \} $ for some norm~\cite{fang_08}.
    If the eigenvalues of  $\tilde{S}_k$ are bounded away from zero independently of $k$ and the Armijo condition is satisfied by backtracking,
    then all limit points of the method using the directions obtained by solving~\eqref{scaled_gradient} with $S_k = \tilde{S}_k$
    are stationary for \eqref{the problem}\cite{bertsekas_99,murray_2011}.
    The most successful and well-established algorithms for the computation of $E_k$ are based on
    modified Cholesky factorizations of the matrix $H(\bx_k)$~\cite{fang_08}.
    Another possibility is setting $E_k = \lambda_k I$, where $\lambda_k > 0$ is a suitable constant, so that the search direction is
    \begin{equation}
    \label{NMdirection}
    \bd_k = -\left(S_k+\lambda_k I\right)^{-1}\bg_k.
    \end{equation}
    
    For quasi-Newton methods ad-hoc globalization strategies have been proposed which avoid matrix factorizations. Next, we briefly
    describe some of them for the BFGS method~\cite{fletcher_2000}. In this case
    \begin{equation}
    \bx_{k+1}=\bx_{k}+ \bd_k, 
    \label{BFGS_iteration}
    \end{equation}
    with $\bd_k$  solution of the system
    \begin{equation}
    B_k \bd=-\bg_k, 
    \label{BFGS_system}
    \end{equation}
    where the matrix $B_k\in \Re^{n \times n}$ is updated by the formula
    \begin{equation}
    B_{k+1} = B_k -\frac{B_k \bs_k \bs_k^\top  B_k}{\bs_k^\top B_k \bs_k}+\frac{\by_k \by_k^\top}{\by_k^\top \bs_k},
    \label{BFGS_update}
    \end{equation}
    with $\by_k=\bg_{k+1}-\bg_{k}$ and $\bs_k=\bx_{k+1}-\bx_{k}$.
    
    For convex optimization problems it can be shown that, under suitable hypotheses, the BFGS method with
    a line search satisfying the Wolfe conditions
    is globally convergent and locally superlinearly convergent~\cite{byrd_1987}.
    For nonconvex functions Dai~\cite{dai_2003} showed with an example that
    the BFGS with Wolfe line search may fail. Later on Mascharenhas~\cite{mascarenhas_04} showed that the BFGS method,
    as well as other methods in the Broyden class, may fail for nonconvex objective functions when an exact line search is used.
    
    For nonconvex minimization problems Li and Fukushima \cite{li_01} proposed a modified version of BFGS, called
    MBFGS, using an Armijio line search or a Wolfe one, and based on an update formula for the matrix in~\eqref{BFGS_system},
    which is equal to~\eqref{BFGS_update} with $\by_k$ replaced by 
    \begin{equation}
    \bar{\by}_k = \bg_{k+1}-\bg_{k}+\gamma_k\|\bg_k\|\bs_k,
    \end{equation}
    where
    \begin{equation}
    \gamma_k= 1 + \max
    \left\lbrace -\frac{\by_k^\top \bs_k}{\|\bs_k\|^2},0 \right\rbrace. 
    \label{MBFGS}
    \end{equation}
    This update formula guarantees that~\cite[Section~5]{li_01}
    $$ 
    \bar{\by}_k^\top \bs_k > \| \bg_k\|\,\|\bs_k\|^2,
    $$
    and therefore
    $B_{k+1}$ is positive definite, provided $B_k$ is positive definite, thus ensuring that the descent condition
    $\bg_{k+1}^\top \bd_{k+1} < 0$ holds. The update formula~\eqref{MBFGS} was inspired by the MN method
    with search direction
    \begin{equation}
    \label{MNdirection}
    \bd_k^{\mathrm{MN}} = -(H_k + \lambda_k I)^{-1}\bg_k.
    \end{equation}
    where $\lambda_k$ is a regularization parameter. The MBFGS method with Armijo or Wolfe line search is globally convergent even
    for nonconvex problems~\cite{li_01}. 
    
    Later on, the same authors proposed the following BFGS formula with \textit{cautious update rule}~\cite{li_01a}:
    \begin{equation}
    B_{k+1}= \left\lbrace
    \begin{array}{ll}
    B_{k+1} = B_k -\frac{B_k \bs_k \bs_k^\top  B_k}{\bs_k^\top B_k \bs_k}+\frac{\by_k \by_k^\top}{\by_k^\top \bs_k}, & \;\, \textrm{if} \; \frac{\by_k^\top \bs_k}{\| \bs_k\|^2 } >\chi \| \bg_k\|^\upsilon,\\
    B_k, & \; \textrm{otherwise,} 
    \end{array}
    \right.
    \label{CBFGS}
    \end{equation}
    where $\chi$ and $\upsilon$ are positive constants, and for which global convergence was proved without convexity assumptions.
    
    In this paper we consider a globalization approach applicable to any Newton-type method. The basic idea consists
    of linearly combining the NT and SD directions. The goal is to bring the iterates sufficiently 
    close to a solution through the globally convergent SD method, so that once the iterates are in the basin of attraction
    of the NT method, it can lead to faster convergence than SD. Although this approach is not new
    (see, e.g., \cite{shi_1996,shi_2000,han_2003} and~\cite[Section~1.4.4]{bertsekas_99}) and much simpler
    than those based on trust regions and incomplete factorizations,
    it has been little utilized, likely because of little confidence in the SD method.
    Here we show that a hybrid strategy which combines SD and NT directions can give very interesting numerical results.
    In our opinion, a suitable scaling of the SD direction is a key issue in making this approach effective. We also find that,
    besides fostering global convergence, this strategy can be effective in speeding up NT methods.
    
    This article is organized as follows. In Section~\ref{sec:SDG} we present our globalization strategy,
    and how the coefficient governing the linear combination can be computed.
    Section~\ref{sec:convergence} deals with the convergence of the resulting algorithm, in particular for Newton's method.
    In Section~\ref{sec:experiments} we discuss results of numerical experiments carried out with our algorithm using either
    Newton's or the BFGS method, including a comparison with some benchmarks algorithms. We conclude
    in Section~\ref{sec:conclusions}.
    
    \section{Globalization strategy\label{sec:SDG}}
    We propose a line-search method of the form~\eqref{linesearch_iteration}, where $\bd_k$ is the NT direction
    if the angle criterion~\eqref{angle_criterion} is satisfied, otherwise
    \begin{equation} \label{new_scaled_gradient}
    \bd_k = \beta_k\,\bd_k^{\mathrm{NT}}-(1- \beta_k)\,\xi_k\,\bg_k,
    \end{equation}
    with $0 \le \beta_k \le 1$ and $\xi_k > 0$.
    A sketch of this method, which we call \emph{SD Globalized (SDG) line-search method}, is provided in Algorithm~\ref{SDG}.
    
    \begin{algorithm}[t!]
        \caption{SD Globalized (SDG) line-search method\label{SDG}}
        \begin{algorithmic}[1]
            \State choose $\bx_0 \in \Re^n$,  $ \{\varepsilon_k\}_{k\in\N} \subset (0,\,1) $, $\sigma_1 \in (0,1/2)$;
            \For {$k=0,1,2,\ldots$}
            \State compute $S_k$;
            \State ${\bd}_k^{\mathrm{NT}}= -S_k^{-1}\bg_k$;
            \If {$ \cos\left\langle \bd_k^{\mathrm{NT}}, -\bg_k \right\rangle \geq \varepsilon_k$}
            \State $\bd_k={\bd}_k^{\mathrm{NT}}$;
            \Else
            \State compute the step length $\xi_k$;
            \State {find $\beta_k$ such that $\cos \left\langle \beta_k \bd_k^{\mathrm{NT}}-(1-\beta_k) \xi_k\,\bg_k, -\bg_k\right\rangle \geq \varepsilon_k$; \label{Alg_SDG_find_beta}}
            \State $\bd_k= \beta_k\,\bd_k^{NT}-(1- \beta_k)\,\xi_k\,\bg_k$;
            \EndIf
            \State select $\alpha_k$ satisfying \eqref{armijo} by backtracking, starting from $ \alpha_k=1$;
            \label{Alg_NSD_backtracking_alpha}	
            \State $\bx_{k+1}=\bx_{k} + \alpha_k \bd_k$;
            \label{Alg_NSD_Armijo_LS}
            \EndFor
        \end{algorithmic}
    \end{algorithm}
    
    Regarding the NT direction~\eqref{scaled_gradient}, which can be, e.g., a Newton, quasi-Newton or inexact Newton direction, we assume that 
    it is well-scaled, so that we scale only the SD direction through $\xi_k$.
    The search direction~\eqref{new_scaled_gradient} is closely related to the one proposed by Shi in \cite{shi_1996,shi_2000}, which is defined as
    \begin{equation}
    \label{shi_direction}
    \bd_k = \beta_k\,\bd_k^{\mathrm{NT}}-(1- \beta_k)\,\bg_k.
    \end{equation}
    Notice that, unlike~\eqref{shi_direction}, the search direction~\eqref{new_scaled_gradient} is invariant to the scaling of the objective function,
    as long as $\xi_k\,\bg_k$ is invariant (e.g., when $\xi_k$ is a BB step length and $\xi_0 = 1 / \| \bg_0 \|$).
    
    The next theorem shows how to compute values of $\beta_k$ guaranteeing that~\eqref{new_scaled_gradient} satisfies~\eqref{angle_criterion}.
    We note that the first part slightly generalizes Lemma~2.1 in~\cite{shi_1996}.
    \begin{theorem}\label{thm:beta_opt_subopt}
        Let us consider any $\bx_k \in \Re^n$ and $\varepsilon_k \in (0,1)$, and assume that
        \begin{equation}
        \label{nodescent}
        -\bg_k^\top \bd^{\mathrm{NT}}_k
        < \varepsilon_k \|\bg_k \|\,\|\bd^{\mathrm{NT}}_k\|,
        \end{equation}
        with $\bd^{\mathrm{NT}}_k$ solution of system~\eqref{scaled_gradient}
        where $S_k$ is any symmetric matrix not multiple of $I$. 
        \begin{itemize}
            \item[i)] Let $\beta_k^\varepsilon$ be the smallest root in $(0,1)$ of the polynomial 
            \begin{equation}
            \label{shi_equation1}
            P_k(\beta)= A_k \beta^2+B_k \beta + C_k,
            \end{equation} 
            where
            \begin{eqnarray}
            A_k &= & 
            \left(\bg_k^\top \bd^{\mathrm{NT}}_k\right)^2 - \varepsilon_k^2 \| \bg_k\|^2\,\|\bd_k^{\mathrm{NT}}\|^2 - B_k - C_k, \nonumber\\
            B_k &= & -2 \left(1-\varepsilon_k^2\right)\xi_k \left\|\bg_k\right\|^2\left(\xi_k \left\|\bg_k\right\|^2 + \bg_k^\top \bd^{\mathrm{NT}}_k\right), \nonumber\\
            C_k &= & \left(1-\varepsilon_k^2\right) \xi_k^2 \left\|\bg_k\right\|^4. \nonumber
            \end{eqnarray}
            If  $\bd_k$ is defined according to \eqref{new_scaled_gradient} with $\beta_k=\beta_k^\varepsilon$, then 
            $$
            -\bg_k^\top \bd_k = \varepsilon_k \|\bg_k \|\,\|\bd_k\|.
            $$
            
            \item[ii)] Let $\bd_k$ be defined according to \eqref{new_scaled_gradient}. Then 
            \begin{equation}\label{angle_criterion_step_k}
            -\bg_k^\top \bd_k \geq \varepsilon_k \|\bg_k \|\,\|\bd_k\|
            \end{equation}
            if and only if $\beta_k\leq\beta_k^\varepsilon$.
            \item[iii)] A lower bound for $ \beta_k^\varepsilon$ is provided by the value $ \hat{\beta}_k$ defined as
            \begin{equation}\label{def:hat_beta_k}
            \hat{\beta}_k= \dfrac{\rho_k}{\rho_k+\pi_k},
            \end{equation}
            where
            \begin{equation}
            \rho_k=\xi_k(1-\varepsilon_k)\quad  \mathrm{and} \quad
            \pi_k = \frac{\bg_k^\top \bd^{\mathrm{NT}}_k}{\|\bg_k\|^2} +
            \varepsilon_k \frac{ \|\bd^{\mathrm{NT}}_k\|}{\|\bg_k\|}.
            \label{def_rho_pi}
            \end{equation}
        \end{itemize}
        
    \end{theorem}
    
    \begin{proof}
        We first prove that $\beta_k^\varepsilon$ is well defined. Let us consider
        \begin{equation*} 
        \Phi_k(\beta)=\frac{-\bg_k^\top\left(\beta\,\bd_k^{\mathrm{NT}}-(1-\beta)\, \xi_k\,\bg_k\right)}
        {\| \bg_k\|\,\| \beta\,\bd_k^{\mathrm{NT}}-(1-\beta)\, \xi_k\,\bg_k \|},
        \end{equation*}
        which is a continuous function of $\beta$. Note that $\Phi_k(\beta)$ is the cosine of the angle
        between the antigradient $-\bg_k$ and the direction
        $$\bd(\beta) = \beta\,\bd_k^{\mathrm{NT}}-(1-\beta) \xi_k\,\bg_k,$$
        which spans continuously the cone between $-\bg_k$, corresponding to $\bd(0)$, and the vector $\bd_k^{\mathrm{NT}}$, corresponding to $\bd(1)$. From its definition it is clear that $\Phi_k(\beta)$ is a monotonically decreasing function in the interval $(0,\,1)$. Since
        $$
        \Phi_k(0) = 1 \;\mbox{ and }\; \Phi_k(1)=\frac{-\bg_k^\top \bd_k^{\mathrm{NT}}}
        {\| \bg_k \|\,\|\bd_k^{\mathrm{NT}}\|} < \varepsilon_k,
        $$
        $\Phi_k(\beta)-\varepsilon_k$ has a unique zero $\beta_k^\varepsilon$ in $(0,1)$.
        The solutions of the equation
        \begin{equation}
        \label{shi_equation}
        \left[ -\bg_k^\top\left(\beta\,\bd_k^{\mathrm{NT}} -(1-\beta)\, \xi_k\,\bg_k\right) \right]^2=
        \left( 
        \varepsilon_k \| \bg_k \|\,\| \beta\,\bd_k^{\mathrm{NT}}-(1-\beta)\, \xi_k\,\bg_k \| \right)^2
        \end{equation} 
        are the solutions of $\Phi_k(\beta)=\pm \varepsilon_k$. By simple computations, it is easy to verify
        that the solutions of~\eqref{shi_equation} are the roots of the polynomial~\eqref{shi_equation1}.
        Now we observe that $P_k(0)= C_k > 0$. To conclude the proof of item i) we need to analyze the two possible cases about the sign of 
        $$
        P_k(1) = \left(\bg_k^\top \bd_k^{\mathrm{NT}}\right)^2-\varepsilon_k^2 \| \bg_k \|^2\,\|\bd_k^{\mathrm{NT}}\|^2.
        $$
        \begin{description}
            \item[{\normalfont a)}] If $-\varepsilon_k <\frac{-\bg_k^\top \bd_k^{\mathrm{NT}}}{
                \|\bg_k\|\,\|\bd_k^{\mathrm{NT}}\| } < \varepsilon_k$, then $P_k(1)<0$ and $\beta_k^\varepsilon$  is the only root of \eqref{shi_equation1} in $(0,1)$;
            \item[{\normalfont b)}] if $\frac{-\bg_k^\top \bd_k^{\mathrm{NT}}}{
                \|\bg_k\|\,\|\bd_k^{\mathrm{NT}}\| } < -\varepsilon_k$, then $P_k(1)> 0$ and $\beta_k^\varepsilon$ is the smallest of the two roots 
            of~\eqref{shi_equation1} in $(0,1)$. 
        \end{description}
        
        Item ii) of the theorem comes from the observation that $\Phi_k(\beta)$ is a monotonically decreasing function in $(0,\,1)$, which
        implies that $\Phi_k(\beta) \geq \Phi_k(\beta_k^\varepsilon) = \varepsilon_k$ for all $\beta \leq \beta_k^\varepsilon $.
        
        To prove item iii), we note that the search direction \eqref{new_scaled_gradient} satisfies \eqref{angle_criterion_step_k} if and only if
        \begin{equation} \label{eqn:hat_beta_cond1}
        \frac{-\beta_k\,\bg_k^\top \bd_k^{\mathrm{NT}} + (1-  \beta_k)\,\xi_k\|\bg_k\|^2}
        {\|{\bg_k}\|\,\|\beta_k \bd_k^{\mathrm{NT}} - (1-  \beta_k)\,\xi_k\,\bg_k \|}  \geq \varepsilon_k.
        \end{equation}
        Since 
        \begin{equation*}
        \frac{-\beta_k \bg_k^\top  \bd_k^{\mathrm{NT}} + (1-  \beta_k)\,\xi_k\|\bg_k\|^2}
        {\|{\bg_k}\|\, \|\beta_k \bd_k^{\mathrm{NT}} - (1-  \beta_k)\,\xi_k\,\bg_k \|}  \geq 
        \frac{-\beta_k\,\bg_k^\top  \bd_k^{\mathrm{NT}} + (1-  \beta_k)\,\xi_k\|\bg_k\|^2}
        {\beta_k \|{\bg_k}\|\,\|\bd_k^{\mathrm{NT}}\| + (1-  \beta_k)\,\xi_k\|\bg_k \|^2},
        \end{equation*}
        a sufficient condition for \eqref{eqn:hat_beta_cond1} to hold is that
        \begin{equation} \label{eqn:hat_beta_cond2}
        \frac{-\beta_k\,\bg_k^\top  \bd_k^{\mathrm{NT}} + (1-  \beta_k)\,\xi_k\|\bg_k\|^2}
        {\beta_k \|{\bg_k}\|\,\|\bd_k^{\mathrm{NT}}\| + (1-  \beta_k)\,\xi_k\|\bg_k \|^2} \geq \varepsilon_k.
        \end{equation}
        This condition gives us a way to compute a lower bound $\hat{\beta}_k$ for $\beta_k^\varepsilon$.
        By straightforward computations one can check that \eqref{eqn:hat_beta_cond2} is equivalent to 
        \begin{equation*}
        \beta_k \left( \xi_k(1-\varepsilon_k) + \frac{\bg_k^\top  \bd_k^{\mathrm{NT}}}{\|\bg_k\|^2} +
        \varepsilon_k \frac{ \|\bd_k^{\mathrm{NT}} \|}{\|\bg_k\|}\right) \leq \xi_k(1-\varepsilon_k), 
        \end{equation*}
        {i.e.,}
        \begin{equation}
        \beta_k \left(\rho_k+\pi_k\right) \leq \rho_k
        \label{beta_condition}
        \end{equation}
        where $\rho_k$ and $\pi_k$ are defined in \eqref{def_rho_pi}.
        We observe that \eqref{nodescent} implies $ \pi_k>0 $ and $\rho_k>0$ comes from the definition of $\xi_k$ and $\varepsilon_k$. 
        Therefore, we have that \eqref{beta_condition} holds if and only if $\beta_k \leq \hat{\beta}_k=\frac{\rho_k}{\rho_k+\pi_k}<1$. Item ii) implies $\hat{\beta}_k \leq \beta_k^\varepsilon $.
    \end{proof} 
    
    \begin{remark}
        Items i) and iii) of the previous theorem suggest two choices for the coefficient $\beta_k$ in \eqref{new_scaled_gradient},
        namely $\beta_k^\varepsilon$ and $\hat{\beta}_k$.  Note that $\beta_k^\varepsilon$ is the largest value of $\beta_k$ such that
        the angle criterion~\eqref{angle_criterion_step_k} is satisfied. Moreover,
        by looking at the definition of $\pi_k$ in \eqref{def_rho_pi} we can easily find a relation between
        the ``quality'' of the NT direction and the value of $\hat{\beta}_k$.
        We can indeed write $\pi_k$ as
        $$
        \pi_k =  \frac{ \|\bd^{\mathrm{NT}}_k\|}{\|\bg_k\|} \left(\frac{\bg_k^\top \bd^{\mathrm{NT}}_k}{\|\bg_k\|\,\|\bd^{\mathrm{NT}}_k\|} +
        \varepsilon_k \right) = -\frac{ \|\bd^{\mathrm{NT}}_k\|}{\|\bg_k\|} \left(\cos\left\langle-\bg_k,\, \bd^{\mathrm{NT}}_k \right\rangle - \varepsilon_k \right),
        $$
        i.e., $\pi_k$ provides a measure of the violation of the angle criterion~\eqref{angle_criterion_step_k}.
        If $\cos\left\langle-\bg_k,\,\bd_k^{\mathrm{NT}}\right\rangle$ approaches $\varepsilon_k$, 
        we have that $\pi_k$ tends to zero and $\hat{\beta}_k$ tends to $1$, allowing us to take a direction close to the NT one.
        Conversely, if $\cos\left\langle-\bg_k,\,\bd_k^{\mathrm{NT}}\right\rangle$ approaches $-1$, the value of $\pi_k$ may increase,
        implying a decrease of $\hat{\beta}_k$ and thus fostering the descent direction to be close to the SD direction.
    \end{remark}
    
    %
    
    Going back to the classical globalization strategies mentioned in the previous section,
    we note that the search direction 
    \begin{equation}
    \label{modified_ssd_direction}
    \bd_k = -\tilde{S}_k^{-1}\bg_k,
    \end{equation}
    where $\tilde{S}_k = S_k+ E_k$, is based on the following quadratic approximation of $f$ at~$\bx_k$:
    \begin{equation}
    \label{quadratic_model}
    \psi_k(\bx) = f(\bx_k) + \bg_k^\top  (\bx-\bx_k) + \frac{1}{2}(\bx-\bx_k)^\top \tilde{S}_k(\bx-\bx_k),
    \end{equation}  
    in which the role of $E_k$ is to guarantee that the model be ``sufficiently'' convex.
    Our approach is based on a different second-order model, namely
    \begin{equation*}
    \phi_k(\bx) = f(\bx_k) + \bg_k^\top  (\bx-\bx_k) + \frac{1}{2}(\bx-\bx_k)^\top W_k(\bx-\bx_k),
    \end{equation*}
    where
    \begin{equation}
    W_k=\left(\beta_k S_k^{-1}+(1-\beta_k)\xi_k I\right)^{-1}.
    \end{equation}
    Even when $W_k$ is not positive definite, the choice of $\beta_k$ guarantees that~\eqref{angle_criterion_step_k} holds for
    $\bd_k = -W_k^{-1}\,\bg_k$. A simple computation shows that we only require convexity for the univariate function 
    \begin{equation}
    \theta(\alpha)= \phi_k\left(\bx_k -\alpha W_k^{-1}\bg_k\right),
    \end{equation}
    which attains its minimum at $\alpha=1$. On the contrary, a globalization strategy like the one based on a shifted linear system of the
    form~\eqref{NMdirection} forces the overall quadratic model \eqref{quadratic_model} to be convex, potentially leading to a model
    insufficiently faithful to $f$.
    
    The directions~\eqref{NMdirection} and~\eqref{new_scaled_gradient} remind us of Trust Region (TR)
    methods~\cite{conn_2000}. These methods compute $ \bx_{k+1} $ by minimizing a quadratic model of $f$ near $\bx_k$,
    \begin{equation}\label{TRmodel}
    \varphi_k(\bx) = f(\bx_k) + \bg_k^\top  (\bx-\bx_k) + \frac{1}{2}(\bx-\bx_k)^\top S_k (\bx-\bx_k), \quad \| \bx-\bx_k \| \leq \Delta_k,
    \end{equation}
    where $\Delta_k$ is updated at each iteration to get a ``good'' approximation of $f(\bx)$ in the ball with center $\bx_k$ and radius 
    $\Delta_k$. The point $ \bx_{k+1} $ is a minimizer of $\varphi_k(\bx)$ in this ball
    if and only if $\bd_k=\bx_{k+1} - \bx_k$ is a solution of the system
    \begin{equation}
    \label{TRsystem}
    \begin{array}{ll}
    (S_k+\lambda I)\,\bd &= -\bg_k,\\ 
    \lambda\,(\Delta_k - \|\bd \|) &= 0,
    \end{array}
    \end{equation} 
    for a scalar $\lambda \ge 0$ such that $S_k+\lambda I \succcurlyeq 0$. Therefore, we can regard direction~\eqref{NMdirection}
    as a TR step. Direction~\eqref{new_scaled_gradient} can be also related to the
    dogleg and the two-dimensional subspace minimization approaches, which provide approximate solutions
    to the TR subproblem~\eqref{TRmodel} by making a search in the space spanned by the SD and Newton directions~(see, e.g., \cite{nocedal_06}).
    
    As already pointed out, unlike \eqref{shi_direction}, the search direction \eqref{new_scaled_gradient} is invariant to the scaling
    of the objective function when $\xi_k\,\bg_k$ is invariant.  In order to clarify the relevance of this issue,
    we show a simple numerical example. We consider the so-called
    ``Brown badly-scaled'' function~\cite{more_81}, defined as
    \begin{equation} \label{eq:bbs}
    f_B(x_1,x_2) = ( x_1 - 10^{6} )^2 + ( x_2 - 2\cdot10^{-6} )^2 + ( x_1 x_2 - 2 )^2.
    \end{equation}
    We compare the performances of the line-search methods
    using as search directions respectively~\eqref{shi_direction} and~\eqref{new_scaled_gradient}
    (with $\xi_k$ set as the BB2 step length defined in~\cite[equation~(5)]{barzilai_borwein_88}),
    on scaled versions of~\eqref{eq:bbs},
    $$
    f_\omega(x_1,x_2) = \omega f_B(x_1,x_2),
    $$
    for different values of the scale factor $\omega$.
    For both algorithms, $\bd^{NT}$ is set as the Newton direction. Furthermore, we set $\varepsilon_k = 10^{-3}$ for each $k$ and
    $ \|\nabla f_\omega (x_1,x_2)\| < 10^{-5} \, \omega $ as stop condition.
    Table~\ref{tab:ShiVSscaling} contains the number of iterations ({its}) and function evaluations ({evals}) performed by the SDG method;
    as expected, when the search directions~\eqref{shi_direction} are used, the performance of the line-search algorithm
    dramatically depends on the scale factor.
    On the other hand, the results in Table~\ref{tab:ShiVSscaling} confirm that the search directions computed using \eqref{new_scaled_gradient}
    are scale invariant. Notice that the Newton direction is scale invariant, whereas the gradient scales with $\omega$, so that the larger the scale factor,
    the closer~\eqref{shi_direction} is to the SD direction. This probably explains why in Table~\ref{tab:ShiVSscaling} we observe a progressive
    deterioration of the performance of the search direction~\eqref{shi_direction} when moving from $\omega = 10^{-3}$ (close to the Newton direction) to 
    $\omega=10^{3}$ (close to the SD direction). Finally, a comparison between the two search directions for $\omega=1$ demonstrates the importance 
    of a suitable choice of~$\xi_k$ in the hybrid strategy (note that in~\eqref{shi_direction} it is always $\xi_k=1$).
    \begin{table}[h!]
        \caption{Comparison of line-search methods with directions \eqref{new_scaled_gradient} and \eqref{shi_direction} on the solution of Brown badly-scaled function with different scalings.}
        \centering
        {
            \medskip
            \begin{tabular}{|l|r|r|r|r|}
                \hline
                \multicolumn{1}{|c|}{} &      \multicolumn{2}{c|}{$\bd_k$ as in \eqref{new_scaled_gradient}}       &         \multicolumn{2}{c|}{$\bd_k$ as in \eqref{shi_direction}}\\ \cline{2-5}
                \multicolumn{1}{|c|}{$\omega$} & \multicolumn{1}{r|}{its} & \multicolumn{1}{r|}{evals} & \multicolumn{1}{r|}{its} & \multicolumn{1}{r|}{evals} \\ \hline
                $10^{-3}$ &                           6 &                            12 &                           6 &                            19 \\
                $10^{-2}$ &                           6 &                            12 &                           6 &                            20 \\
                $10^{-1}$ &                           6 &                            12 &                          10 &                            32 \\
                $1$ &                           6 &                            12 &                          15 &                            42 \\
                $10$ &                           6 &                            12 &                          19 &                            46 \\
                $10^{2}$ &                           6 &                            12 &                          34 &                            86 \\
                $10^{3}$ &                           6 &                            12 &                         224 &                          1683 \\ \hline
            \end{tabular}
            \label{tab:ShiVSscaling}
        }
    \end{table}
    
    \section{Convergence \label{sec:convergence}}
    
    We now focus on the convergence properties of the SDG method.
    The theorems in this section are a slight modification of results presented by the authors in~\cite{diserafino_2020lncs}.
    
    \begin{theorem}
        Let $ f\in C^2(\Re^n)$ and assume that $\{\varepsilon_k\}$ is bounded away from $0$. Then, for
        any $\bx_0$, the limit points of the sequence $\{\bx_k\}$ generated by Algorithm~\ref{SDG} are stationary.
    \end{theorem}
    \begin{proof}
        Let $0<\varepsilon_{min} =\inf\{\varepsilon_k\}$. At any iteration $k$, Theorem~\ref{thm:beta_opt_subopt} guarantees
        the existence of a coefficient $\beta_k$ for step~\ref{Alg_SDG_find_beta} of Algorithm~\ref{SDG}, therefore a direction
        $\bd_k$ satisfying $ \cos \left\langle \bd_k,\,-\bg_k\right\rangle \geq \varepsilon_k $
        can be found. Since the sequence $\left\lbrace \varepsilon_k\right\rbrace$ is bounded from below by
        $ \varepsilon_{min} $ and $\alpha_k$ is obtained by a backtracking technique to fulfill condition~\eqref{armijo},
        the thesis follows from~\cite[Proposition 1.2.1]{bertsekas_99}.
    \end{proof}
    
    The next theorem shows that the SDG method has quadratic convergence rate when the direction
    $ \bd_k^{\mathrm{NT}} $ is the Newton direction.
    The proof is omitted because it is practically the same as the proof of Theorem~2 in~\cite{diserafino_2020lncs}.
    
    \begin{theorem}
        Let $ f\in C^2(\Re^n)$ and let $\{\bx_k\}$ be generated by the SDG method where $ \bd_k^{\mathrm{NT}} $
        is the Newton direction. Let $\sigma_1$ in~\eqref{armijo} be such that $0 < \sigma_1 < \frac{1}{2}$ and the
        sequence $\{\varepsilon_k\}$ be nonincreasing with limit $\varepsilon_{min} > 0 $. Suppose also that there
        exists a limit point $\hat{\bx}$ of $\{\bx_k\}$ where $H(\hat{\bx})$ is positive definite and $H(\bx) $ is Lipschitz
        continuous around $\hat{\bx}$. If $\varepsilon_{min}$ is sufficiently small, then $\{\bx_k\}$  converges to
        $\hat{\bx}$ with quadratic rate.
    \end{theorem}

    \section{Computational experiments\label{sec:experiments}}
    
    We implemented two MATLAB versions of the SDG algorithm, using Newton's method and the BFGS method, and compared 
    them with the Modified Newton method using a modified Cholesky factorization~\cite{fang_08} and with the CBFGS method~\cite{li_01a},
    respectively. To better understand the effect of our globalization strategy, we also run Newton's method,
    the BFGS method, and the SD one with the BB2 Barzilai-Borwein step length used in the numerical example
    at the end of Section~\ref{sec:SDG}.
    
    In the SDG and pure SD methods, we set $\xi_0 = 1 / \| \bg_0 \|$, $\xi_k = \max \{ \xi_k^{\mathrm{BB2}}, \nu_1 \}$ if $\xi_k^{\mathrm{BB2}} > 0$,
    and $\xi_k = \min \{ 10 \, \xi_{k-1}^{\mathrm{BB2}}, \nu_2 \}$ otherwise; here $\xi_k^{\mathrm{BB2}}$ is the BB2 step length, $\nu_1 = 10^{-5}$ and
    $\nu_2 = 10^5$. We chose BB2 instead of other BB step lengths (see, e.g., \cite{diserafino_2018}) because BB2 was more effective in preliminary
    numerical experiments. The Hessian approximations in SDG with BFGS, in BFGS and in CBFGS were initialized as explained
    in~\cite[page~143]{nocedal_06}.
    We applied a shrinking strategy for the selection of $\varepsilon_k$ in~\eqref{angle_criterion_step_k}: given
    $\varepsilon_0 \in (0,1)$ and $\zeta = 0.95$, at the $k$-th iterate ($k>0$) we set $\varepsilon_k = \zeta \, \varepsilon_{k-1}$
    if $\beta_{k-1} < 1$, and $\varepsilon_k = \varepsilon_{k-1}$ otherwise. To prevent the sequence $\{\varepsilon_k\}$
    from going toward zero, we set a threshold for $\varepsilon_k$ equal to $\bar\varepsilon = 10 \, \varepsilon_{mac}$,
    where $\varepsilon_{mac}$ is the machine epsilon. It is worth noting that in none of the tests performed the value of $\varepsilon_k$
    reached the threshold. In all the algorithms, the Armijo backtracking line search with $\sigma_1=10^{-4} $ (see~\eqref{armijo})
    and quadratic interpolation~\cite[Section~3.5]{nocedal_06}  was performed. The methods were stopped as soon as
    \begin{equation}\label{eqn:stopping_criterion}
    \| \bg_k \| < \tau_g \| \bg_0 \|,
    \end{equation}
    with $\tau_g = 10^{-5} $;
    as a safeguard we also stopped the execution when a maximum number, $k_{max}$, of 2000 iterations was achieved or the objective function
    appeared to get stuck, i.e.,
    $$
    \frac{f(\bx_{k-1})-f(\bx_k)}{|f(\bx_{k-1})|} < \bar\varepsilon.
    $$
    
    Algorithm~\ref{SDG_implemented} is a detailed version of Algorithm~\ref{SDG} that includes implementation details.
    By numerical experiments we found that the use of the pure SD direction with BB2 step length is computationally
    convenient when $\bd_k={\bd}_k^{\mathrm{NT}}$ is not a descent direction (see lines 12-13).
    The notation SDG[NT,$\varepsilon_0$] highlights that the algorithm uses a selected NT method (e.g., Newton's or BFGS) and
    $\varepsilon_0$ as initial value of the sequence $\{ \varepsilon_k \}$.
    
    All the experiments were carried out using MATLAB R2018b. 
    Comparisons were performed by using the performance profiles introduced in~\cite{dolan_2002}, which are briefly described next for completeness.
    
    Let $\mathcal{S}_{\mathcal{T}\!,\,\mathcal{A}}\ge 0$
    be a statistic corresponding to the solution of a test problem $\mathcal{T}$ by an algorithm $\mathcal{A}$, and suppose
    that the smaller the statistic the better the algorithm. Furthermore, let $\mathcal{S}_{\mathcal{T}}$ be the smallest value attained
    on the test $\mathcal{T}$ by one of the algorithms under analysis.
    The performance profile of the algorithm $\mathcal{A}$ is defined as     
    $$   
    \pi(\chi)=\frac { \mbox{number of tests such that }
        \mathcal{S}_{\mathcal{T}\!,\,\mathcal{A}} /\mathcal{S}_{\mathcal{T}} \le \chi }
    { \mbox{number of tests} }, \quad \chi \ge 1,
    $$
    where the ratio $\mathcal{S}_{\mathcal{T},\,\mathcal{A}} / \mathcal{S}_{\mathcal{T}}$ is set to
    $+\infty$ if algorithm $\mathcal{A}$ fails in solving~$\mathcal{T}$.
    In other words, $\pi(\chi)$ is the fraction of problems for which $\mathcal{S}_{\mathcal{T},\,\mathcal{A}}$
    is within a factor~$\chi$ of the smallest value $S_{\mathcal{T}}$.
    Thus $\pi (1)$ is the percentage of problems for which $\mathcal{A}$ is the best,
    while $\lim_{\chi\rightarrow +\infty} \pi(\chi)$ gives the percentage of problems that are successfully
    solved by $\mathcal{A}$.
    
    The performance profiles considered in this work use as performance statistics
    the number of iterations and the number of function evaluations. We note that in Section~\ref{sec:numerical_res_Mod_Chol},
    in comparing our SDG algorithm based on Newton's method with an MN method, we do not consider the execution time
    because the MN implementation exploits a C code for the modified Cholesky factorization, called via a MATLAB mex file,
    while the Newton systems in the SDG algorithm are solved by the MATLAB function \textsf{backslash}.
    Thus, in this case a time comparison would be unfair. In all the other experiments
    the algorithms compared have about the same cost per iteration, therefore using the number of iterations as performance statistic
    appears sensible. Furthermore, the number of objective function evaluations is related to the number of line searches performed,
    and provides also information on the quality of the descent direction. Nevertheless, we use a performance profile based on
    the execution time at the end of Section~\ref{sec:numerical_res_Mod_Chol}, to further support some results.
    
    To better show the effectiveness of the proposed strategy, we considered two sets of test problems, described in the following section.
    
    \begin{algorithm*}[t!]
        \caption{SDG[NT,$\varepsilon_0$] \label{SDG_implemented}}
        \begin{algorithmic}[1]
            \State choose an NT method and $ \varepsilon_0 \in (0,1) $;
            \State choose $\bx_0 \in \Re^n$, $\zeta \in (0,1)$, $\sigma_1 \in (0,1/2)$, $ \bar\varepsilon = 10 \, \varepsilon_{mac}$, $ \tau_g \in (0,1)$, $ k_{\max} \in \N$;
            \State $k=0;$ \textit{continue} = true;
            \While {($\|\bg_k\| \geq \tau_g \|\bg_0\|$) $\wedge$ \textit{continue}}
            \State compute $S_k$ according to the NT method;
            \State ${\bd}_k^{\mathrm{NT}}= - S_k^{-1}\bg_k$;
            \If {$ \cos\left\langle \bd_k^{\mathrm{NT}},\,-\bg_k \right\rangle \geq \varepsilon_k$}
            \State $\bd_k={\bd}_k^{\mathrm{NT}}$;
            \State $ \varepsilon_{k+1} = \varepsilon_{k} $;
            \Else
            \State compute the step length $\xi_k$;
            \If {$ \cos\left\langle \bd_k^{\mathrm{NT}},\,-\bg_k \right\rangle \leq 0$}
            \State $\bd_k= -\xi_k \bg_k$;
            \Else
            \State {compute $\hat{\beta}_k$ as defined in \eqref{def:hat_beta_k}; \label{SDG_implemented_compute_hat_beta_k}}
            \State $\bd_k= \hat{\beta}_k\bd_k^{\mathrm{NT}}-(1-\hat{\beta}_k) \xi_k\,\bg_k$;
            \EndIf
            \State $ \varepsilon_{k+1} = \max \left\lbrace \bar\varepsilon,\; \zeta \, \varepsilon_{k}\right\rbrace $;
            \EndIf
            \State select $\alpha_k$ satisfying \eqref{armijo} by backtracking with quadratic interpolation \cite[Section~3.5]{nocedal_06}, starting from $ \alpha_k=1$;
            \State $\bx_{k+1}=\bx_{k} + \alpha_k \bd_k$;
            \State $k = k+1$;
            \If{$ \left(k > k_{max}\right) \, \vee\, \left( \vert f_{k-1}- f_{k}\vert < \bar\varepsilon \,\vert f_{k-1}\vert  \right) $} 
            \State \textit{continue} = false;
            \EndIf
            \EndWhile
        \end{algorithmic}
    \end{algorithm*}
    
    \subsection{Test problems}
    
    \subsubsection{Nonconvex problems}
    
    We considered 36 problems available from \url{https://people.sc.fsu.edu/~jburkardt/m\_src/test\_opt/test\_opt.html},
    including the unconstrained minimization problems from the Mor\'{e}-Garbow-Hillstrom collection~\cite{more_81}
    and other problems. We set the problem size equal to 100 for all the problems where the dimension could be chosen
    by the user. For each problem we used 10 starting points, i.e., the point $\bx_0$ provided with the problem and the points
    $\bx_0^s$, with $s=1,\ldots,9$, where $(\bx_0^s)_i = (\bx_0)_i + \gamma_i^s$, $\gamma_i^s$ was
    a random number in $\left[-\eta_s a_i, \eta_s a_i \right]$, $a_i = \left| (\bx_0)_i \right|$, and the values
    $ \eta_s $ were logarithmically spaced in the interval $ \left[10^{-2},10^{-1}\right] $.
    These choices resulted in a set of 360 nonconvex optimization problem instances.

    \subsubsection{Convex problems coming from machine learning}
    
    The second set of test problems consists in the minimization of convex functions arising from machine learning. In particular, given $N$ pairs
    $ (\ba_i,b_i) $, where $\ba_i \in \Re^n$ and $b_i \in \{-1,1\}$, we considered the problem of training a linear classifier by minimizing the function
    \begin{equation}\label{eqn:logistic_regression_problem}
    f(\bx)=\frac{1}{N} \sum_{i=1}^{N} f_i(\bx) + \frac{\mu}{2} \|\bx\|^2,    
    \end{equation}
    where $f_i(x) = \log\left(1+ e^{-b_i\,\ba_i^\top\bx}\right)$ and $ \mu > 0 $.
    
    \begin{table}[h!]
        \caption{Number of points and number of features for each machine learning dataset.
            We indicate the source by adding a superscript to the dataset name,
            according to the following list: \newline
            1 - \url{https://www.csie.ntu.edu.tw/~cjlin/libsvmtools/datasets/},\newline
            2 - \url{http://www.ics.uci.edu/~mlearn/MLRepository.html},\newline
            3 - NAACCR Incidence - CiNA Public File, 1995-2015, North American Association of Central Cancer Registries,\newline
            4 - \url{http://yann.lecun.com/exdb/mnist}.}
        \centering
        {
            \medskip
            \begin{tabular}{|l|r|r|}
                \hline
                \multicolumn{1}{|c|}{name} & \multicolumn{1}{c|}{points} & \multicolumn{1}{c|}{features} \\ \hline
                a6a\footnotemark[1]        &                          11220 &                              123 \\
                a7a\footnotemark[1]        &                          16100 &                              123 \\
                a8a\footnotemark[1]        &                          22696 &                              123 \\
                a9a\footnotemark[1]        &                          32561 &                              123 \\
                adult\footnotemark[2]      &                          48842 &                              122 \\
                cina\footnotemark[3]       &                          16033 &                              132 \\
                cod-rna\footnotemark[1]    &                          59535 &                                8 \\
                ijcnn1\footnotemark[1]     &                          49990 &                               22 \\
                mnist\footnotemark[4]      &                           7603 &                              100 \\
                mushrooms\footnotemark[1]  &                           8124 &                              112 \\
                phishing\footnotemark[1]   &                          11055 &                               68 \\
                w6a\footnotemark[1]        &                          17188 &                              300 \\
                w7a\footnotemark[1]        &                          24692 &                              300 \\
                w8a\footnotemark[1]        &                          49749 &                              300 \\ \hline
            \end{tabular}
            \label{tab:datasets}
        }
    \end{table}
    
    We considered 14 datasets, whose dimensions and sources are reported in Table~\ref{tab:datasets}.
    For each dataset we considered a 10-fold cross validation setting, thus obtaining 10 different training problems of the form \eqref{eqn:logistic_regression_problem} with $N$ approximately equal to 0.9 times the total number of points. For each problem we set
    $\mu=\frac{1}{N}$, which is a choice usually found in literature. This produced a total of 140 instances for training our strategy.
    For this set of test problems we focused on the BFGS method.

    \subsection{Numerical results}
    
    \subsubsection{Comparison on the choice of $\beta_k$}
    First, we focused on the choice of $\beta_k$ in~\eqref{new_scaled_gradient}.
    Theorem~\ref{thm:beta_opt_subopt} suggests $\beta_k^\varepsilon$ and $\hat{\beta}_k$ as two possible alternatives. Just to get a first picture,
    we considered the so-called ``Gulf research and developement'' function \cite{more_81}, defined as
    \begin{equation*}
    f_{GRD}(x_1,x_2,x_3) = \sum_{i=1}^{99} \left[ \exp\left(-\frac{\left\vert \left(-50\log(\frac{i}{100})\right)^{\frac{2}{3}} + 25 - x_2 \right\vert^{x_3}}{x_1}\right) - \frac{i}{100}  \right]^2,
    \end{equation*}
    with starting point $ \bx_0 = [40, 20, 1.2]^\top$. We ran SDG[Newton,0.5] with $\beta_k=\hat{\beta}_k$ and computed also
    $\beta_k^\varepsilon$ at each iteration. The values of $\beta_k^\varepsilon$ and $\hat{\beta}_k$ are shown in the top plot in  Figure~\ref{fig:optimalVSsuboptimal_GR&D}, for the iterations in which the Newton step was rejected as a search direction.
    For the same iterations, in the bottom plot we depicted the values of 
    $\cos\langle\bd_k,\,-\bg_k\rangle$ for $\bd_k$ in~\eqref{new_scaled_gradient}, computed with 
    $ \beta_k= \beta_k^\varepsilon$, $\beta_k=\hat{\beta}_k$ and $\beta_k=1$ (the last one corresponds to the pure Newton's method). 
    This example suggests that the difference between $ \beta_k^\varepsilon$ and  $\hat{\beta}_k$ is negligible, especially when close to the solution.
    Conversely, the angle between $\bd_k$ (computed with either $ \beta_k^\varepsilon$ or $\hat{\beta}_k$) and $\bd_k^{\textrm{Newton}}$
    is non-negligible, especially far from the solution.
    \begin{figure}[h!]
        \centering
        \includegraphics[width=0.9\textwidth]{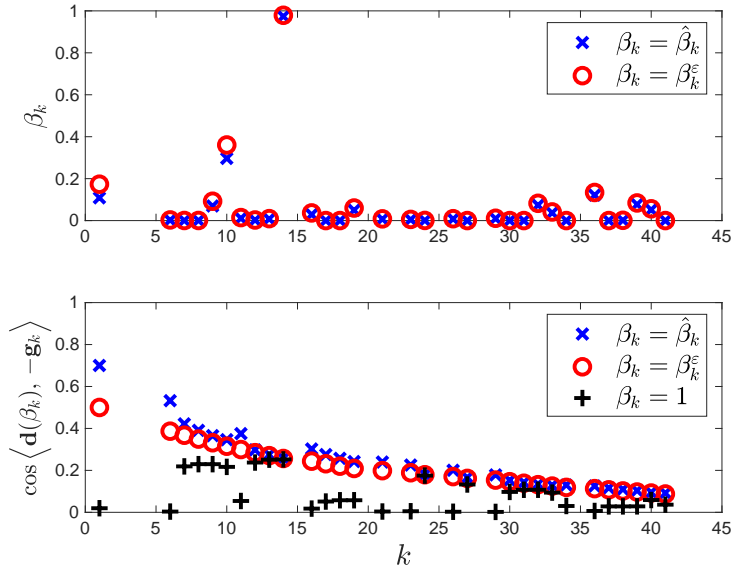}
        \vskip -9pt
        \caption{Test on the Gulf research and development function. Top plot: values of $ \beta_k^\varepsilon$ and $\hat{\beta}_k$. Bottom plot: values of 
            $\cos\langle\bd(\beta_k),\,-\bg_k\rangle$, with $\bd(\beta) = \beta\,\bd_k^{\textrm{Newton}}-(1-\beta) \xi_k\,\bg_k$ for $ \beta_k= \beta_k^\varepsilon$, $\beta_k=\hat{\beta}_k$ and $\beta_k=1$. 
            \label{fig:optimalVSsuboptimal_GR&D}}
    \end{figure}
    
    Then we ran two versions of SDG[Newton,0.5], with $ \beta_k= \beta_k^\varepsilon$ and $\beta_k=\hat{\beta}_k$, on the solution
    of the 360 nonconvex problem instances previously described, looking for experimental evidence about the choice of $\beta_k$. 
    Since the problems are nonconvex, different algorithms may reach different local minima starting from the same point.
    We noted that, out of the 360 considered problem instances, SDG went to the smallest local minimum 267 times with
    $\beta_k=\hat{\beta}_k$, and 295 times with  $ \beta_k= \beta_k^\varepsilon$. 
    To have a fair picture, we compared the two versions of SDG on the
    202 instances where they reached equal solutions (two solutions were considered equal if they coincided up to the third 
    significant digit). 
    The performance profiles reported in Figure~\ref{fig:optimalVSsuboptimal_ALL_newton} show
    a comparison in terms of number of iterations (left) and number of objective function evaluations (right), and suggest that
    $\beta_k=\hat{\beta}_k$ is preferable to $\beta_k^\varepsilon$. If we had to venture a guess,
    based on our experience, we would say that far from the solution
    $\hat{\beta}_k$ can be significantly smaller then $\beta_k^\varepsilon$, and this increases the SD component
    in~\eqref{new_scaled_gradient}. Far from the solution, the SD direction with a suitable step length like BB2 can
    even be more effective than Newton's method in decreasing the objective function, and this might explain, to some extent,
    the results in Figure~\ref{fig:optimalVSsuboptimal_ALL_newton}. 
    
    \begin{figure}[t!]
        \centering
        \includegraphics[width=0.5\textwidth]{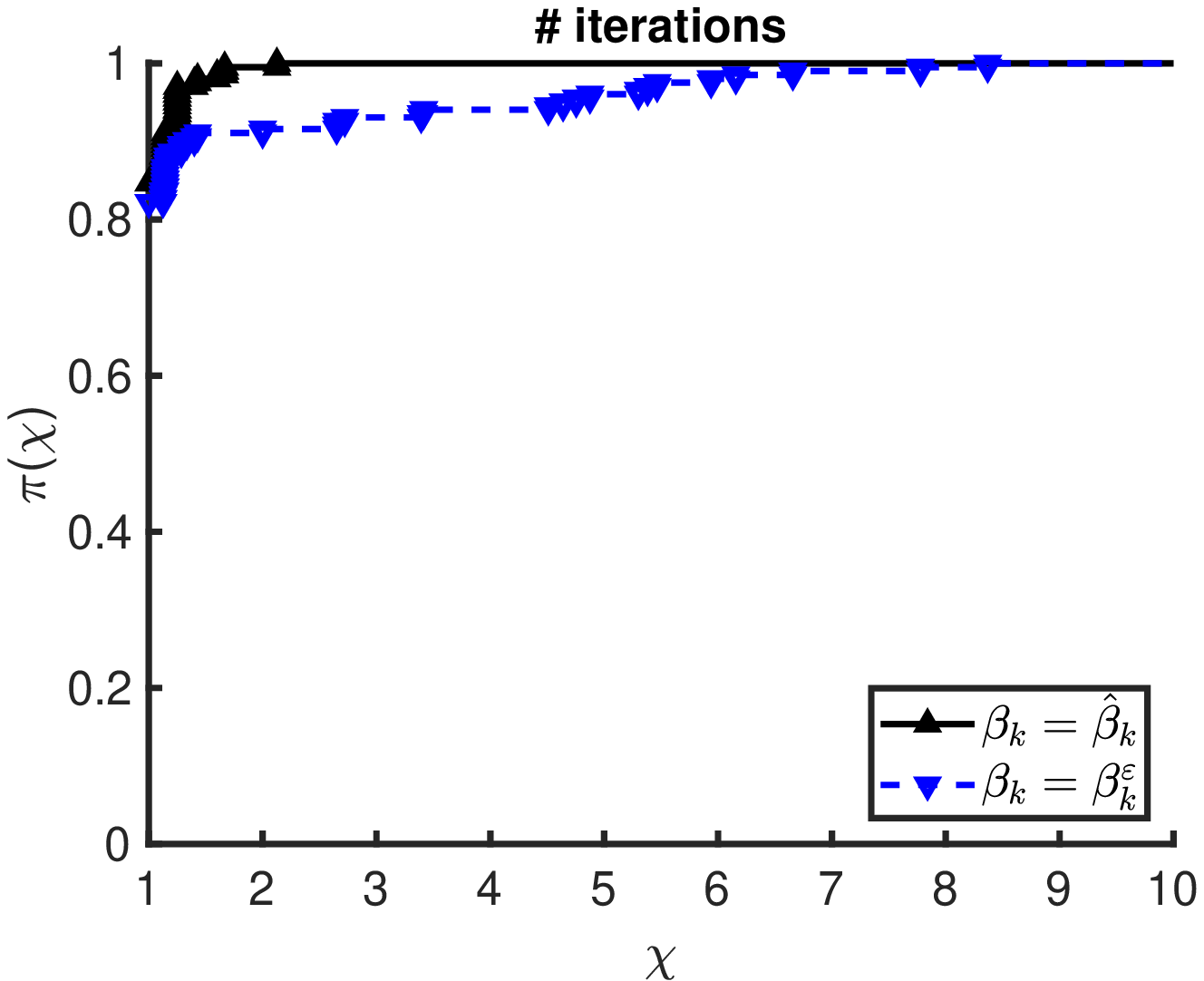}\includegraphics[width=0.5\textwidth]{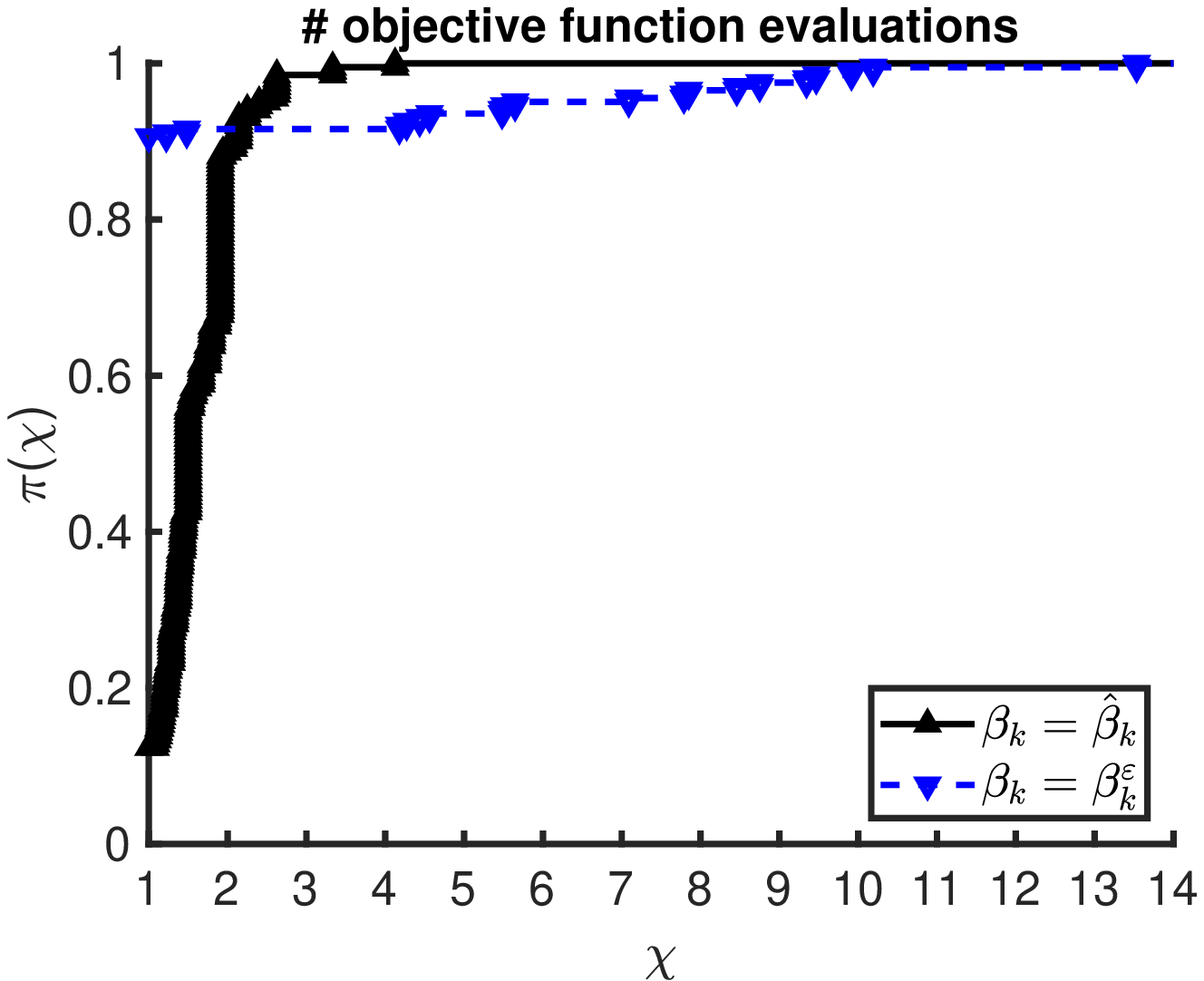}
        \vskip -9pt
        \caption{Performance profiles of SDG[Newton,0.5], with $ \beta_k= \beta_k^\varepsilon$ and $\beta_k=\hat{\beta}_k $, on the solution of the 202 nonconvex problem instances in which the two algorithms reached the same solution.\label{fig:optimalVSsuboptimal_ALL_newton}}
    \end{figure}
    
    We also compared SDG[BFGS,0.5] with $\beta_k^\varepsilon$ and $\hat{\beta}_k$ on the first set of test problems. In this case,
    the two versions of SDG computed the same solution on 148 problem instances.
    As shown by the performance profiles in Figure~\ref{fig:optimalVSsuboptimal_ALL_bfgs}, the implementation with $\beta_k=\hat{\beta}_k$ slightly
    outperformed the one with $\beta_k={\beta}_k^\varepsilon$.
    \begin{figure}[h!]
        \centering
        \includegraphics[width=0.5\textwidth]{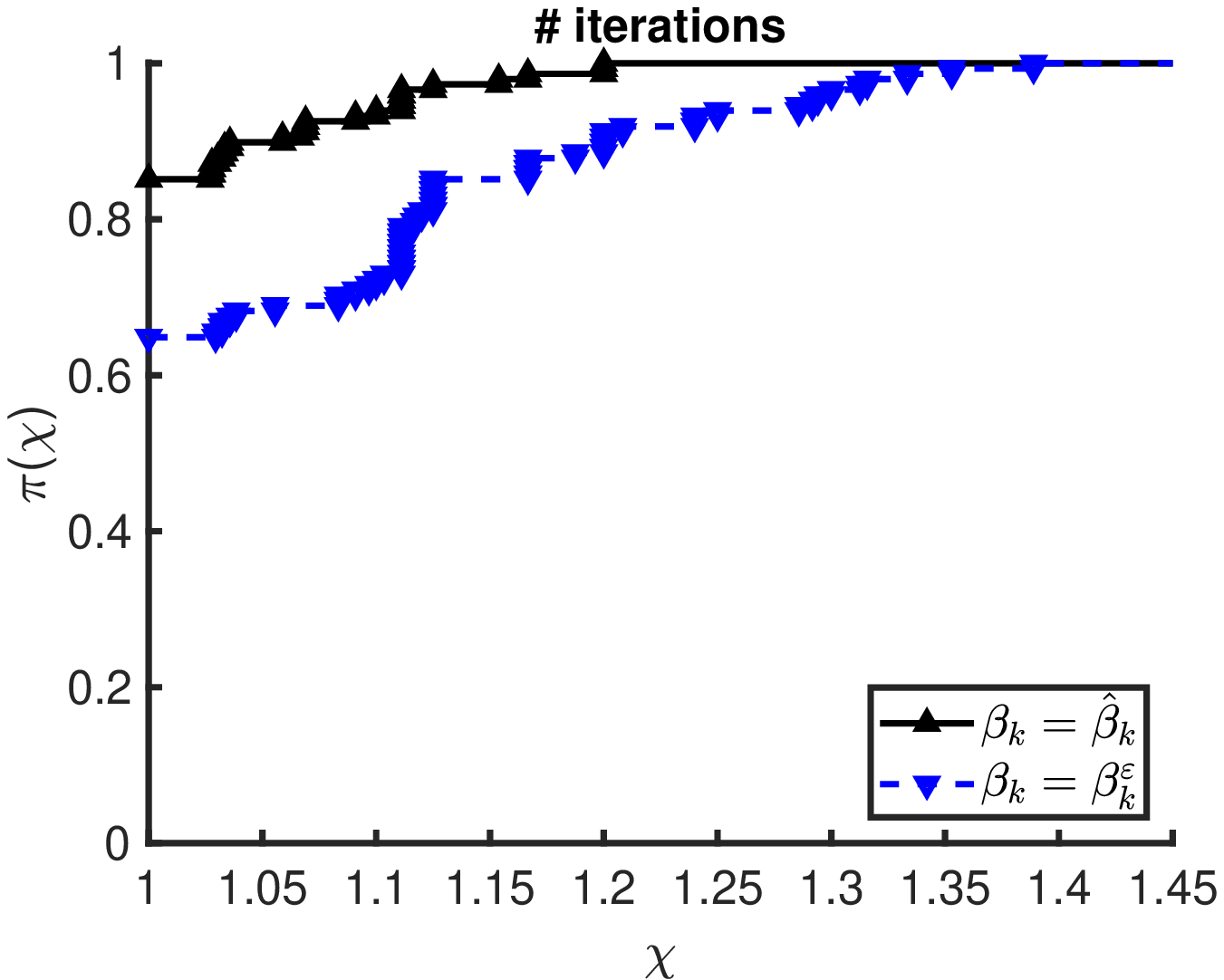}\includegraphics[width=0.5\textwidth]{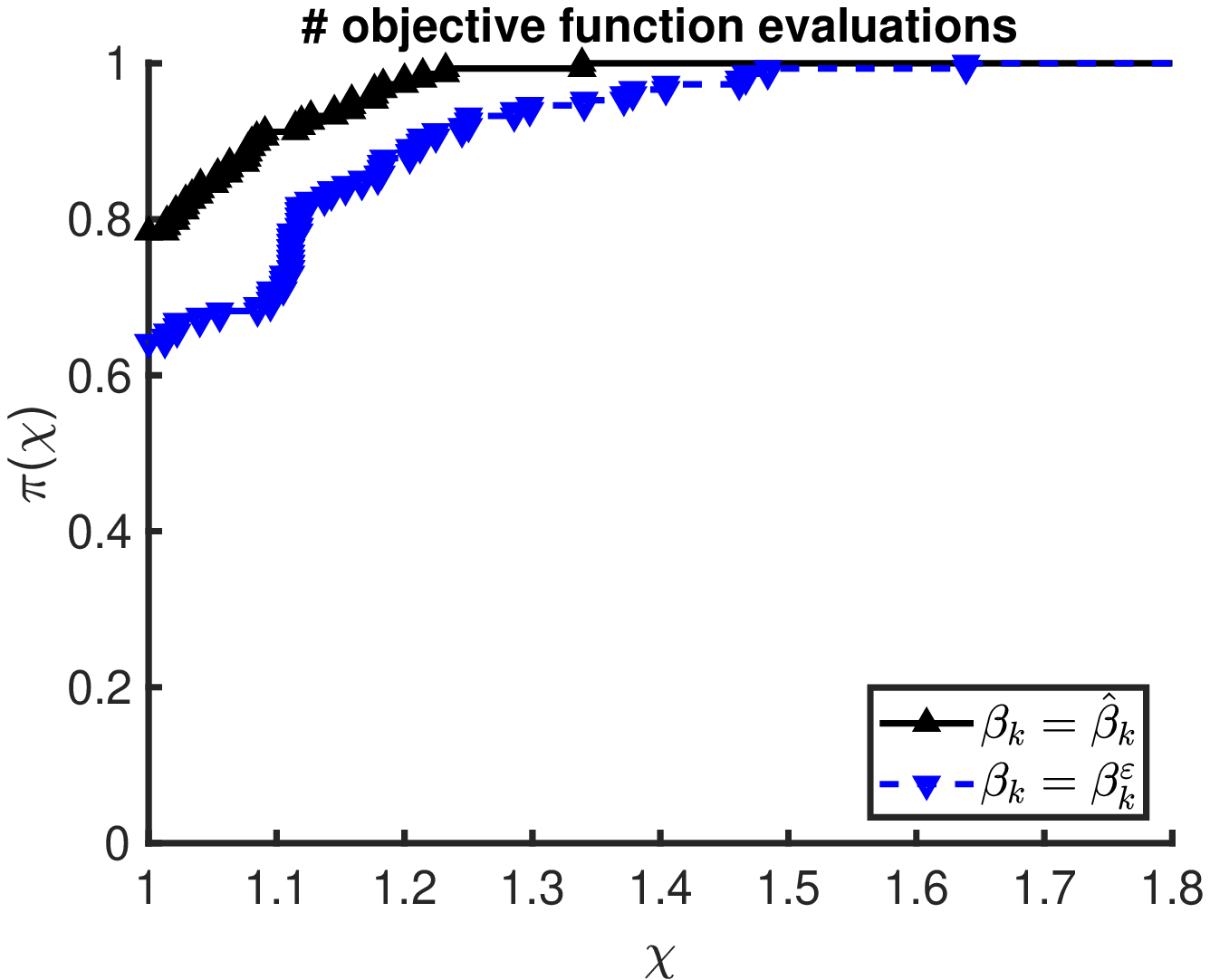}
        \vskip -9pt
        \caption{Performance profiles of SDG[BFGS,0.5], with $ \beta_k= \beta_k^\varepsilon$ and $\beta_k=\hat{\beta}_k $, on the solution of the 148 nonconvex problem instances in which the two algorithms reached the same solution.\label{fig:optimalVSsuboptimal_ALL_bfgs}}
    \end{figure}
    A similar analysis was carried out for the convex problems from machine learning,
    comparing the versions SDG[BFGS,0.5] with the two different values of $\beta_k$. Again, $\beta_k=\hat{\beta}_k$ seems
    to provide the best results, in terms of both number of iterations and number of function evaluations, as shown
    in Figure~\ref{fig:optimalVSsuboptimal_class_05}. Therefore, we decided to set $\beta_k=\hat{\beta}_k$ in the remaining
    numerical experiments. 
    \begin{figure}[h!]
        \centering
        \includegraphics[width=0.5\textwidth]{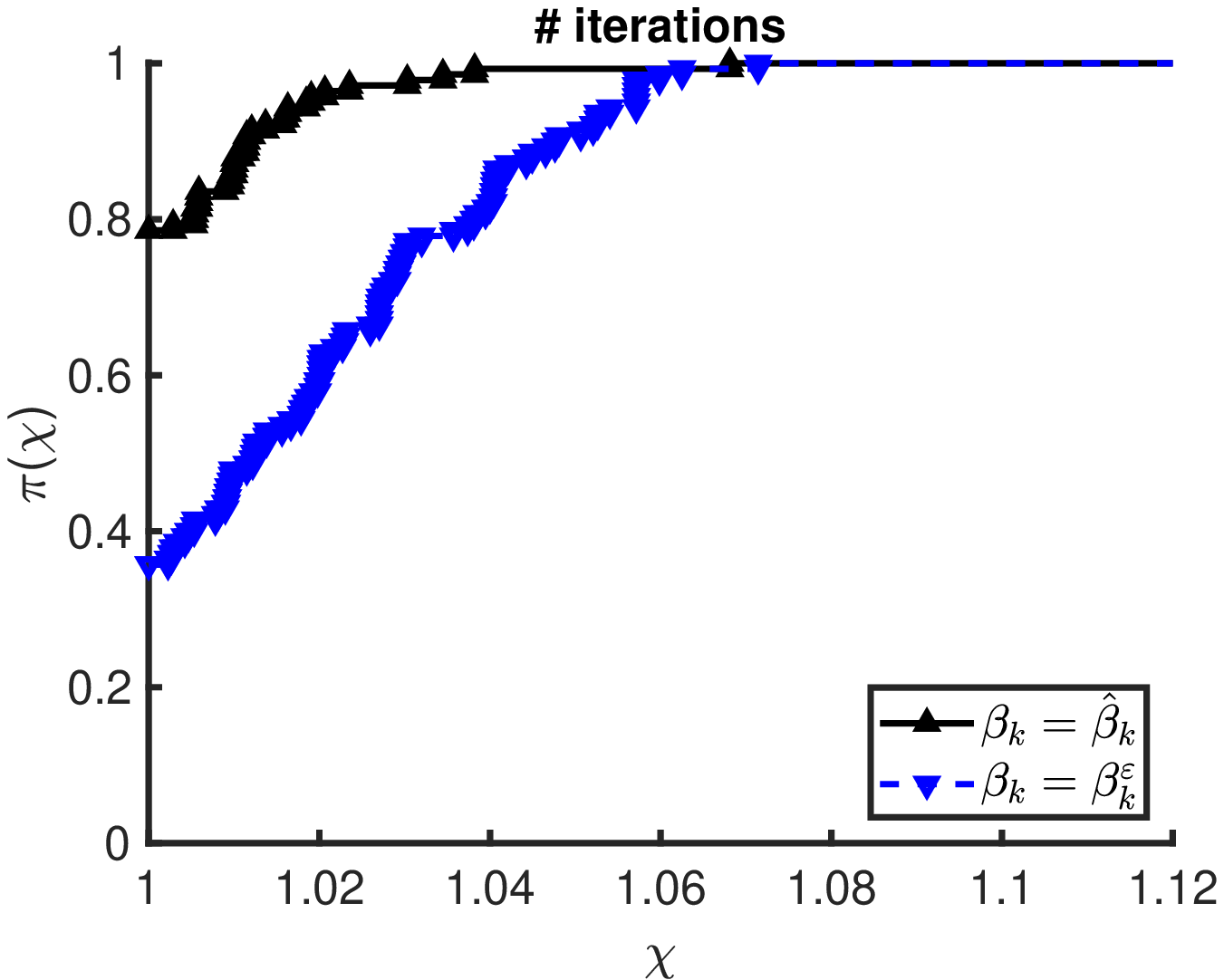}\includegraphics[width=0.5\textwidth]{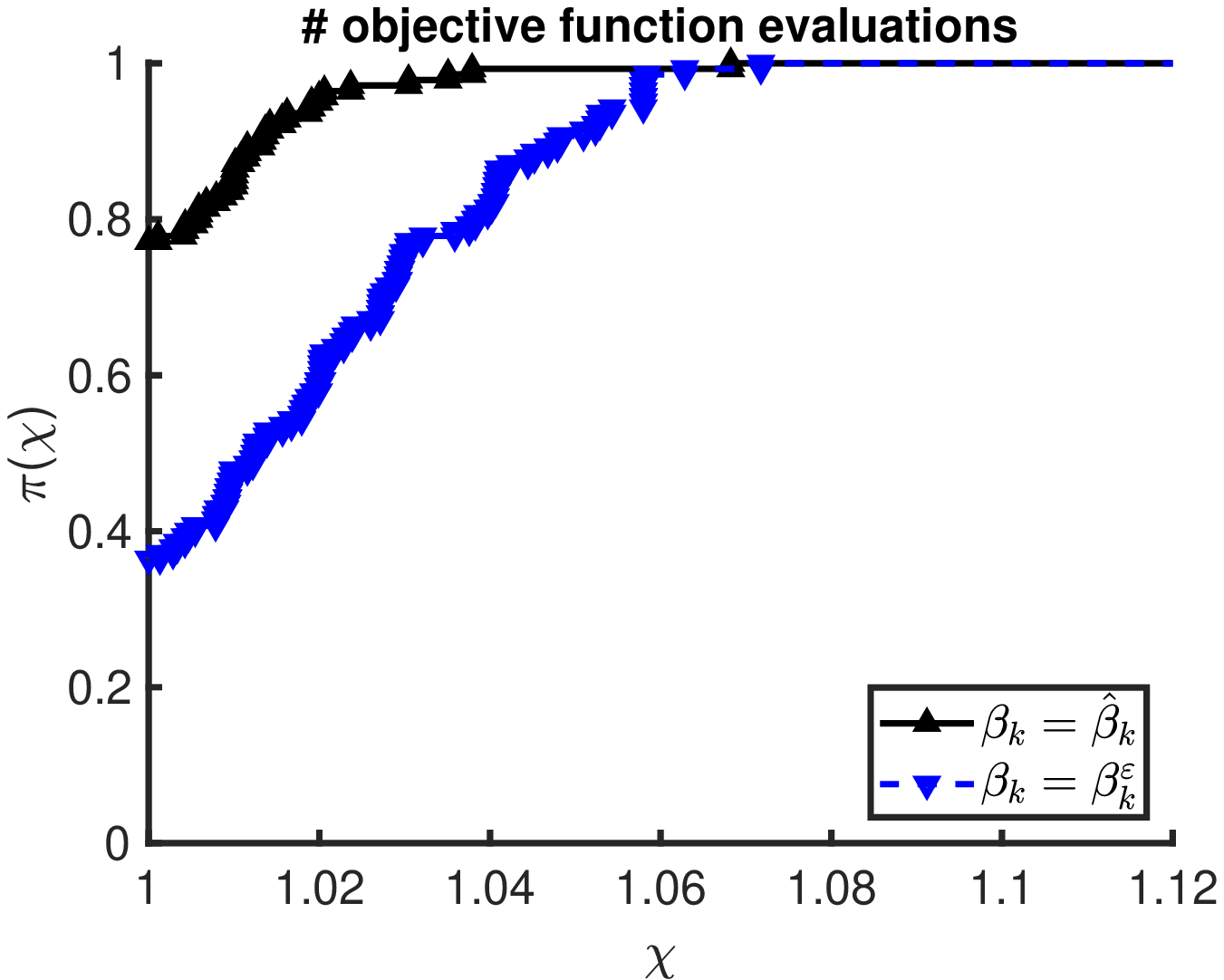}
        \vskip -9pt
        \caption{Performance profiles of SDG[BFGS,0.5], with $ \beta_k= \beta_k^\varepsilon$ and $\beta_k=\hat{\beta}_k $, on the solution of the 140 convex problem instances.\label{fig:optimalVSsuboptimal_class_05}}
    \end{figure}

    \subsubsection{Comparison with other globalization strategies \label{sec:numerical_res_Mod_Chol}}
    
    To perform a comparison with other globalization strategies,
    we ran, on the nonconvex problems, the SDG[Newton,0.5] method and an MN method
    based on the modified Cholesky factorization GMW-II~\cite{fang_08} (see \url{https://github.com/hrfang/mchol}).
    For completeness, we also ran Newton's method.
    
    We found that Newton's method with line search stopped without satisfying criterion~\eqref{eqn:stopping_criterion}
    for 168 out of 360 problem instances.  Conversely, MN failed only on 10 instances, whereas SDG[Newton,0.5]
    was always able to satisfy~\eqref{eqn:stopping_criterion} within 2000 iterations.
    Figure~\ref{fig:nonconvex_FINAL_newton} summarizes the results of the comparison between SDG[Newton,0.5]
    and the Modified Newton's method for the 134 problem instances in which the methods obtained the same solution.
    The profiles show that our algorithm required less function evaluations, although it was slightly less efficient in terms of iterations.
    \begin{figure}[h!]
        \centering
        \includegraphics[width=0.5\textwidth]{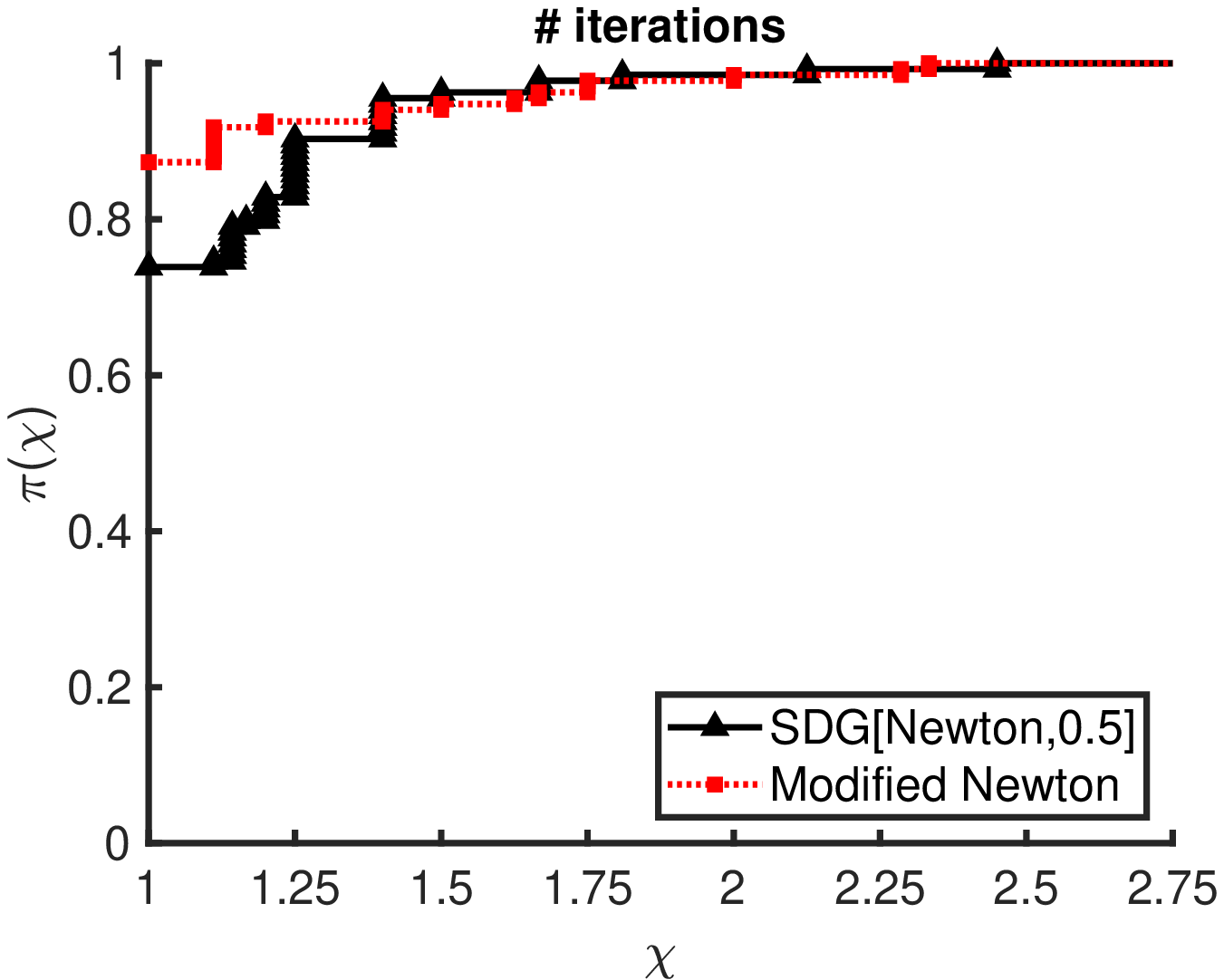}\includegraphics[width=0.5\textwidth]{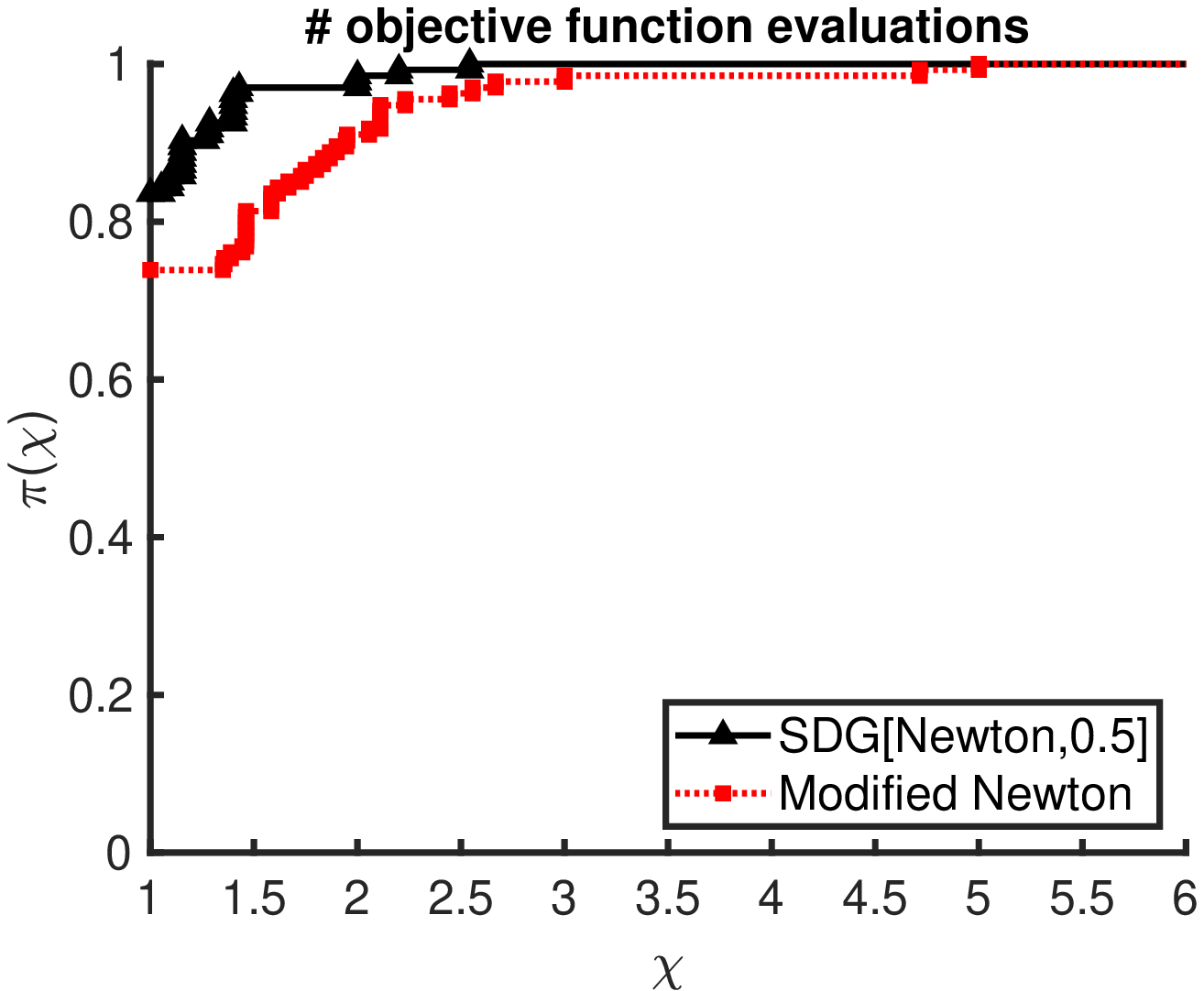}
        \vskip -9pt
        \caption{Performance profiles of SDG[Newton,0.5] and the Modified Newton's method on the solution of the 134 nonconvex problem instances in which they computed the same solution.\label{fig:nonconvex_FINAL_newton}}
    \end{figure}
    
    We also compared SDG[BFGS,0.5] with CBFGS using $ \upsilon=1 $ (see~\eqref{CBFGS}) and with BFGS.
    The performance profiles in Figure~\ref{fig:nonconvex_FINAL_bfgs} show how SDG[BFGS,0.5] compares with CBFGS and BFGS
    in the solution of the 221 problem instances in which the algorithms get the same solution.
    Note that CBFGS and BFGS overlap extensively. In other words, BFGS does not seem to really need
    a globalization strategy, and the cautious update rule~\eqref{CBFGS} is likely to reduce to the standard BFGS
    update rule almost always. Figure~\ref{fig:nonconvex_FINAL_bfgs} also shows that our globalization strategy can slightly
    improve the performance of the BFGS method.
    \begin{figure}[h!]
        \centering
        \includegraphics[width=0.5\textwidth]{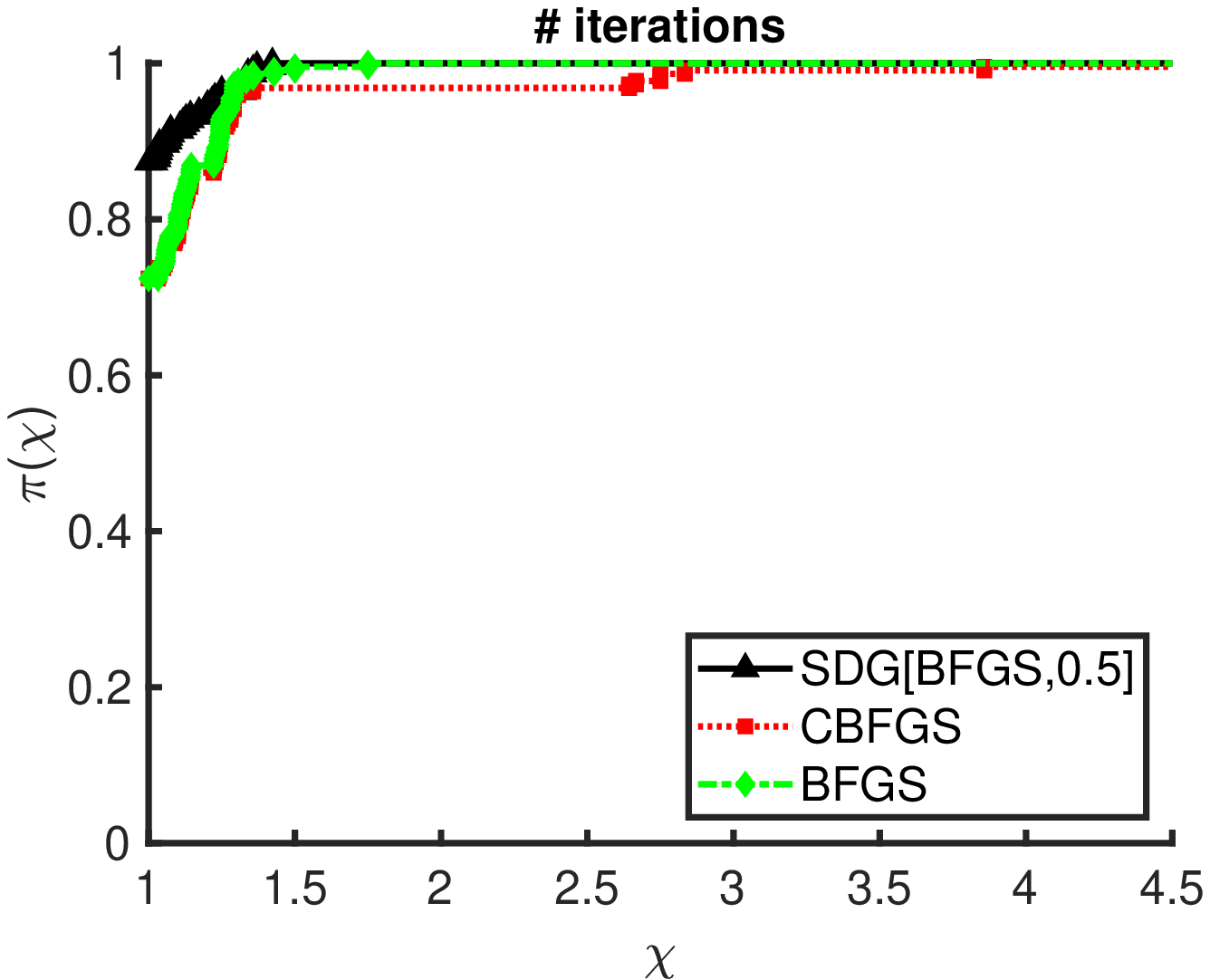}\includegraphics[width=0.5\textwidth]{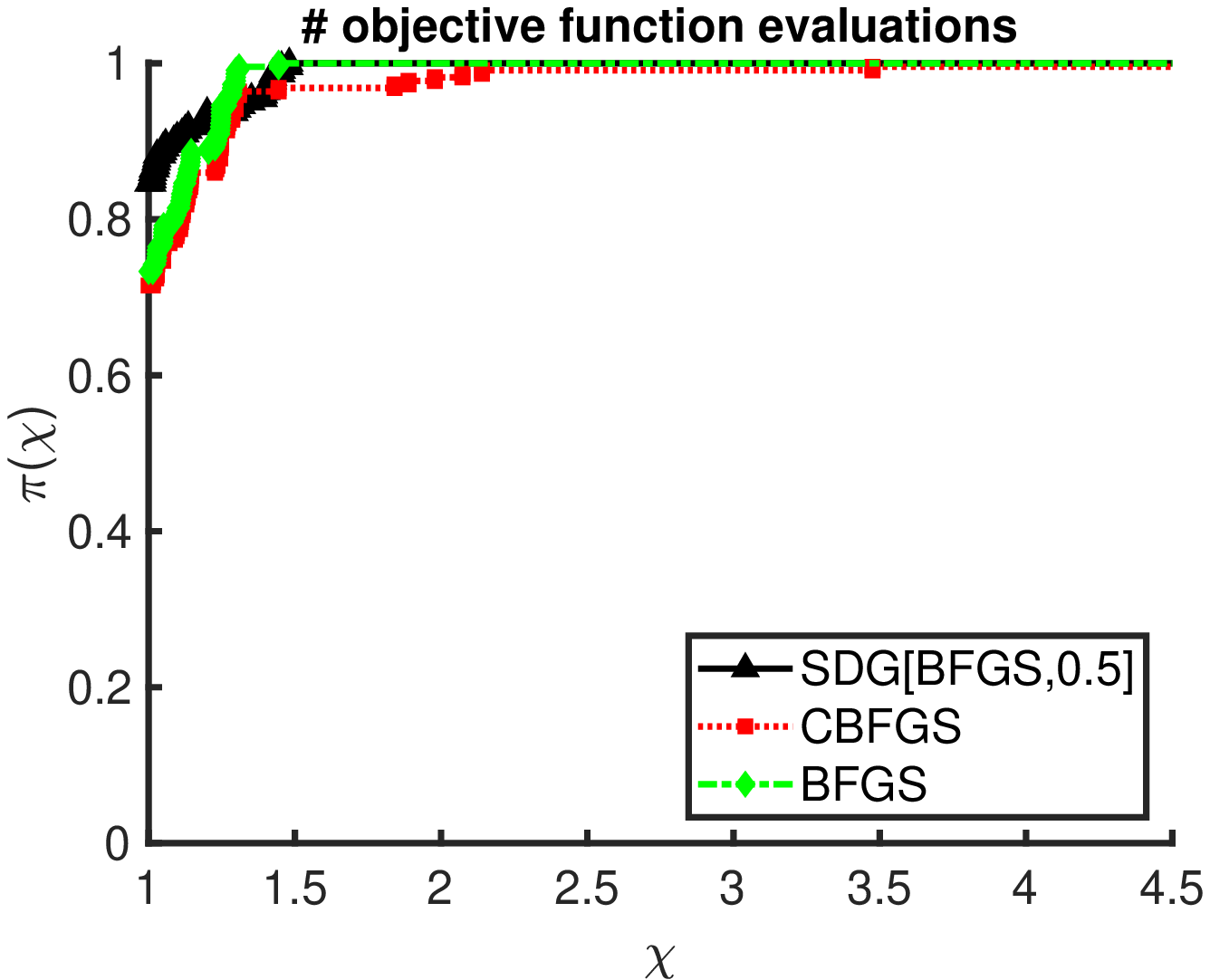}
        \vskip -9pt
        \caption{Performance profiles of SDG[BFGS,0.5], CBFGS and BFGS in the solution of the 221 nonconvex problem instances in which the algorithms computed the same solution.\label{fig:nonconvex_FINAL_bfgs}}
    \end{figure} 
    
    With the aim of better understanding the behavior of SDG[BFGS,0.5],
    in Figure~\ref{fig:betaplot} we plotted the sequence $\{\beta_k \}$ for two representative instances of nonconvex problems.
    We set $\beta_k =1$ when the BFGS direction was accepted (see lines~7-8 of Algorithm~\ref{SDG_implemented}), and $\beta_k = 0$ when
    the SD direction scaled by the BB2 step length was selected (see lines~12-13 of Algorithm~\ref{SDG_implemented}).
    \begin{figure}[h!]
        \centering
        \includegraphics[width=0.9\textwidth]{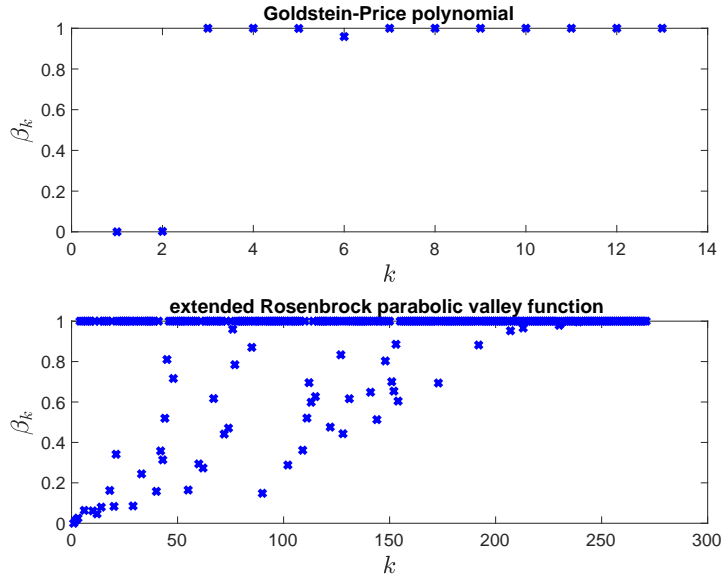}
        \vskip -18pt
        \caption{Values of $ \beta_k $ used by SDG[BFGS,0.5] in the solution of two selected nonconvex problems. \label{fig:betaplot}}
    \end{figure}
    In the top plot, concerning a very easy problem (the Goldstein-Price polynomial), we see that the method practically switches from
    SD to BFGS at the third iteration. The bottom plot, concerning the extended Rosenbrock parabolic valley function, shows that
    in the very first iterations it is likely that $\beta_k$ is close to 0, and the SD component is dominating the
    search direction~\eqref{new_scaled_gradient}. As the number of iterations increases, the SD component in~\eqref{new_scaled_gradient}
    becomes smaller and smaller, and eventually the method reduces to BFGS.
    
    Concerning the convex problems, we considered BFGS only, i.e., we compared
    SDG[BFGS,0.5], CBFGS and BFGS. Figure~\ref{fig:final_bfgs_class_05} confirms the trend already observed
    in the nonconvex case: the CBFGS and BFGS methods behave the same way, and SDG[BFGS,0.5] outperforms both of them.
    \begin{figure}[h!]
        \centering
        \includegraphics[width=0.5\textwidth]{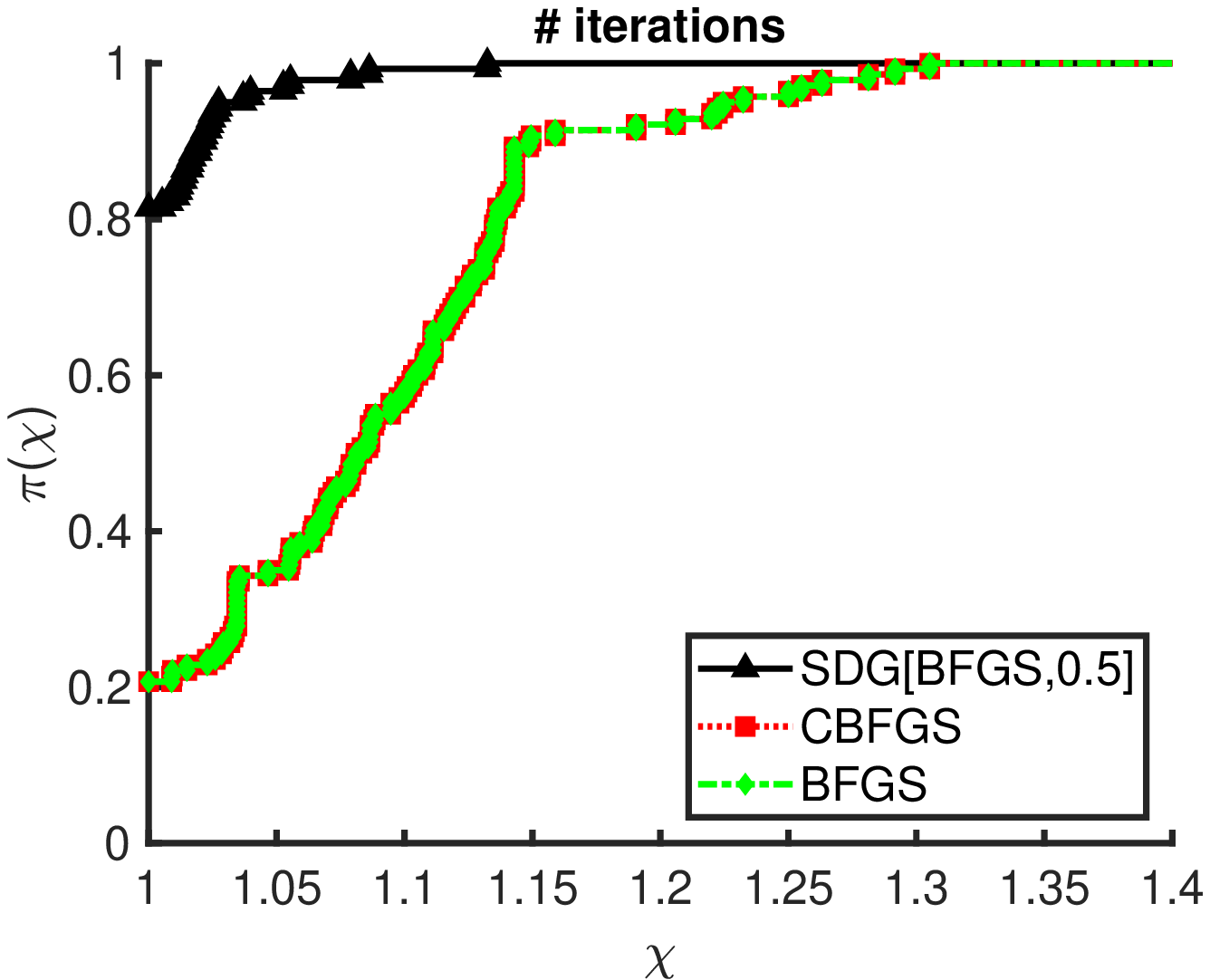}\includegraphics[width=0.5\textwidth]{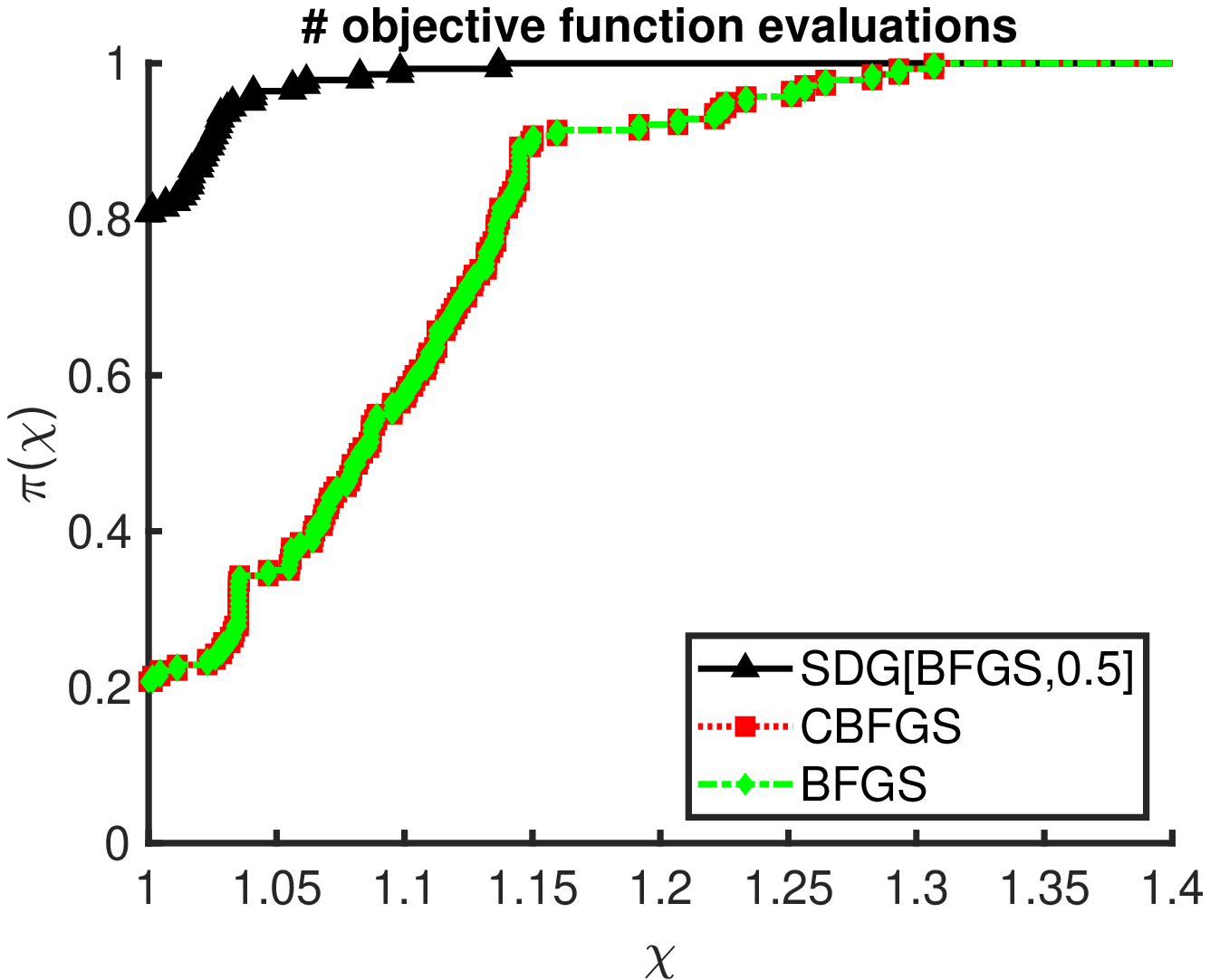}
        \vskip -9pt
        \caption{Performance profiles of SDG[BFGS,0.5], CBFGS and BFGS in the solution of the 140 convex problem instances.\label{fig:final_bfgs_class_05}}
    \end{figure}
    
    The results in the convex case suggest that a suitable linear combination of an NT direction with the SD one
    can have a beneficial effect in speeding up the convergence, in addition to providing global convergence.
    To further investigate this issue, we also made
    computational experiments with SDG[BFGS,0.9] on the convex test problems. Of course, the choice 
    $\varepsilon_0=0.9$ favors the SD component in the search direction, and we cannot suggest it as a safe choice in general.
    However, the comparison with SDG[BFGS,0.5] in Figure~\ref{fig:bfgs_09_vs_05} shows that for the selected problems
    SDG[BFGS, 0.9] is more efficient than SDG[BFGS,0.5],and hence than the standard BFGS.
    \begin{figure}[h!]
        \centering
        \includegraphics[width=0.5\textwidth]{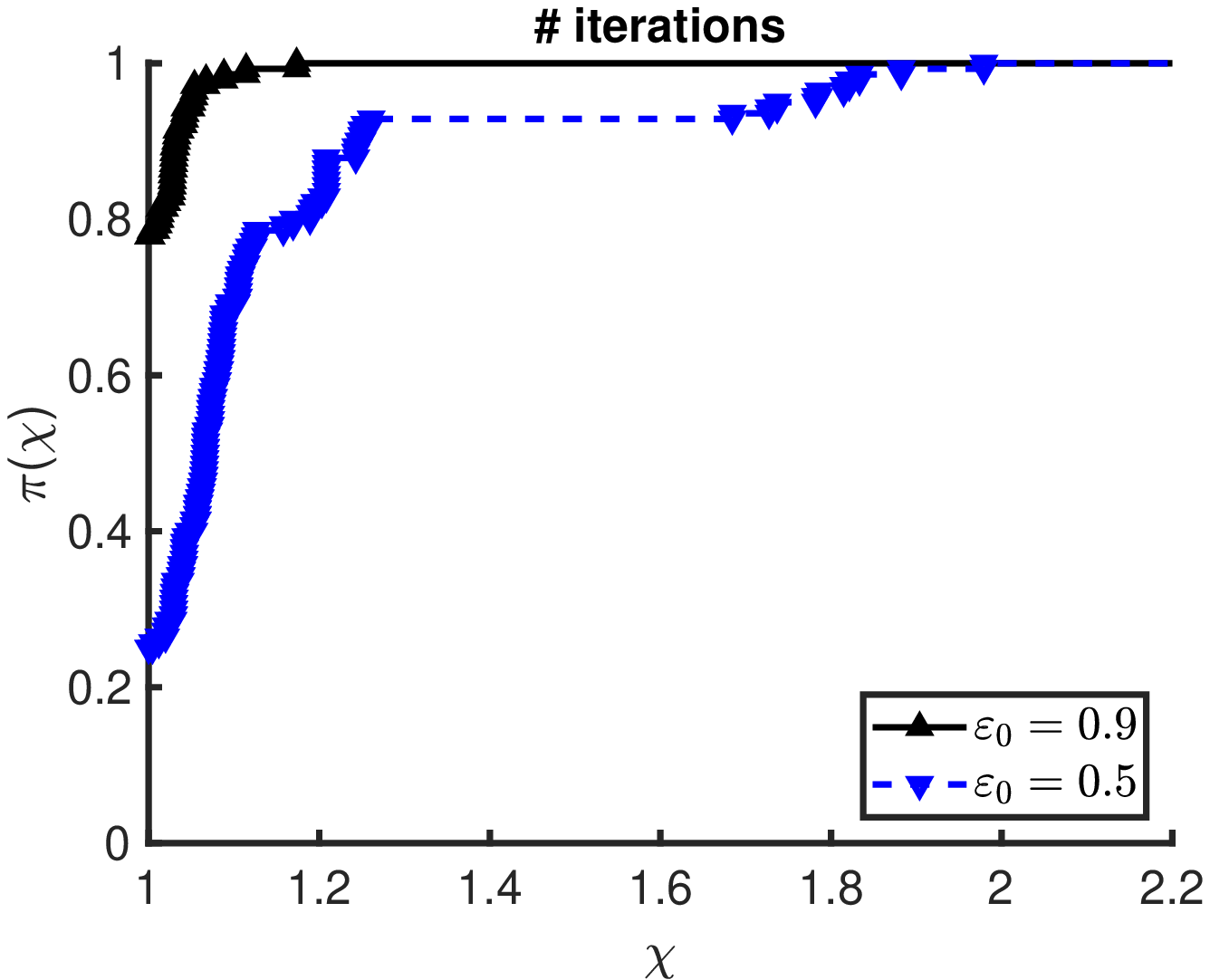}\includegraphics[width=0.5\textwidth]{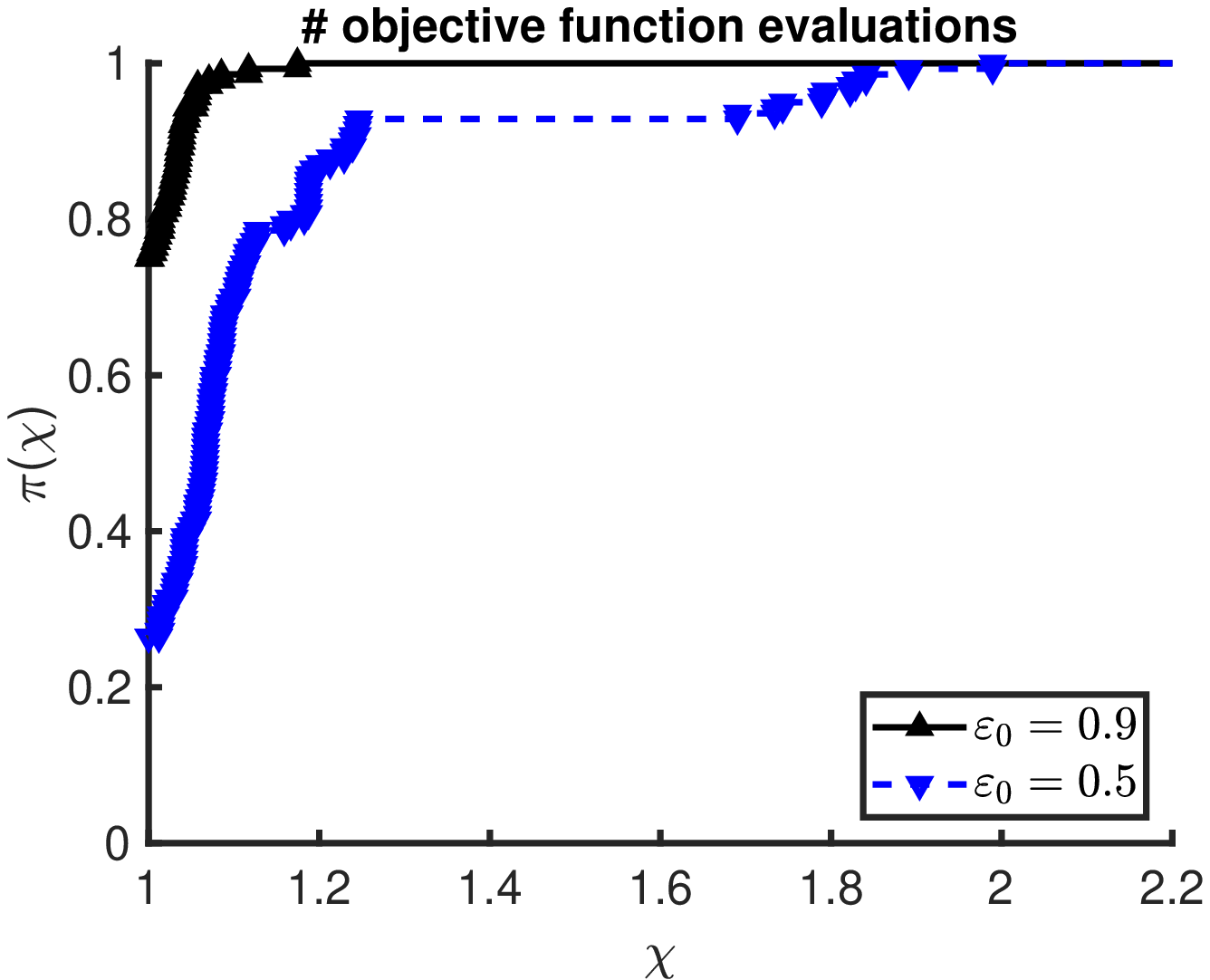}
        \vskip -9pt
        \caption{Performance profiles of SDG[BFGS,0.9] and SDG[BFGS,0.5] in the solution of the 140 convex problem instances.\label{fig:bfgs_09_vs_05}}
    \end{figure}
    \begin{figure}[h!]
        \centering
        \includegraphics[width=0.5\textwidth]{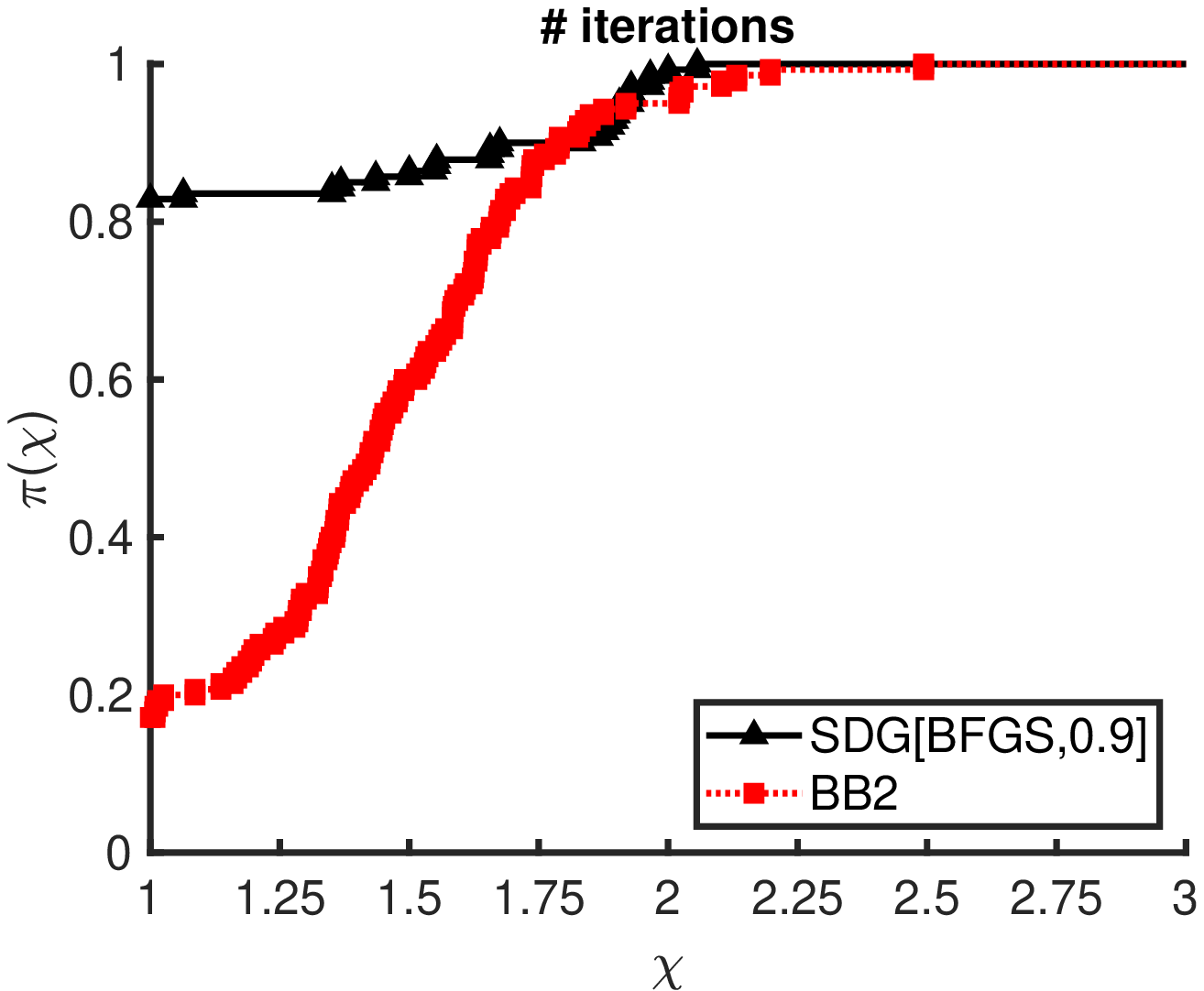}\includegraphics[width=0.5\textwidth]{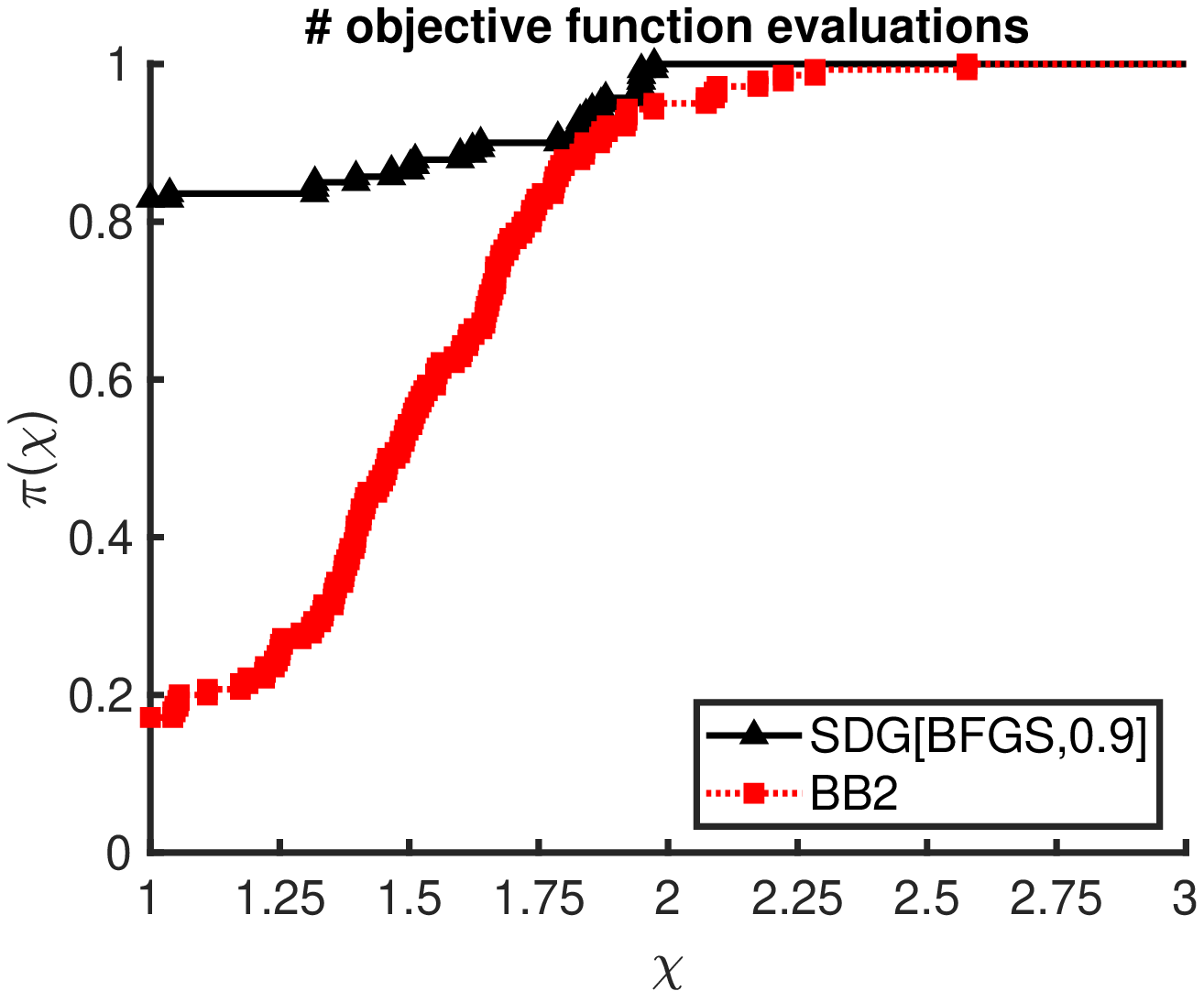}
        \vskip -9pt
        \caption{Performance profiles of SDG[BFGS,0.9] and SD with BB2 step length on the solution of the 140 convex problem instances. \label{fig:final_bfgs_class_09}}
    \end{figure}
    \begin{figure}[h!]
        \centering
        \includegraphics[width=0.5\textwidth]{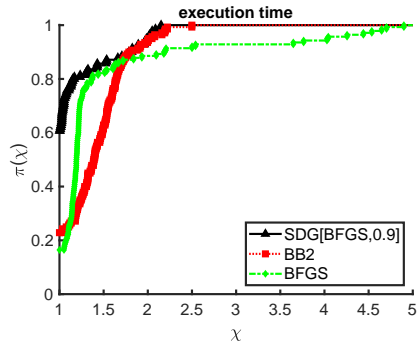}
        \vskip -9pt
        \caption{Performance profiles (execution time) of SDG[BFGS,0.9], BFGS and SD with BB2 step length on the solution of the 140 convex problem instances. \label{fig:final_bfgs_class_09_time}}
    \end{figure}
    
    Finally, the performance profiles in Figure~\ref{fig:final_bfgs_class_09} show that SDG[BFGS,0.9] generally outperforms
    the SD method with BB2 step length. This suggests that the good behavior of SDG[BFGS,0.9] does not depend only on the use
    of SD directions with effective step lengths, but also on the efficient combination of these directions with BFGS ones.
    This is confirmed by Figure~\ref{fig:final_bfgs_class_09_time}, where SDG[BFGS,0.9], BFGS and the SD method with BB2 step length are compared in terms of execution time, showing that the proposed algorithm
    is more efficient than the others.
    
    \section{Conclusions\label{sec:conclusions}}
    
    We proposed a globalization strategy to be used with any NT method, which is based on a linear combination
    of the NT and SD search directions.
    Our approach, which generalizes the one proposed in~\cite{shi_1996,shi_2000},
    looks easier and more flexible than globalization strategies which have been devised ad hoc for specific methods~\cite{fang_08,li_01a}.
    We believe that a key issue in our strategy is to take the SD direction with a suitable step length. The reason is twofold:
    first, from the theoretical point of view, it allows us to have search directions that are invariant to the scaling of the objective function;
    second, it allows us to inject in the globalized method the proven effectiveness of gradient methods based on particular step-length
    rules~\cite{diserafino_2018}.
    Our computational experiments suggest that the use of a line search along a suitable
    linear combination of NT and SD directions can improve numerical performance with respect to the NT method, in addition to providing
    global convergence. In particular, the SD component with the BB2 step length showed a beneficial effect especially when far from the solution.
    
    
    \section*{Acknowledgments}
    This work was supported by GNCS-INdAM, Italy, and by the V:ALERE Program of the University of Campania ``L. Vanvitelli'' under the VAIN-HOPES Project.


\begin{thebibliography}{10}
        
        \bibitem{antonelli_2016}
        {\sc L.~Antonelli, V.~De~Simone, and D.~di~Serafino}, {\em On the application
            of the spectral projected gradient method in image segmentation}, Journal of
        Mathematical Imaging and Vision, 54 (2016), pp.~106--116.
        
        \bibitem{barzilai_borwein_88}
        {\sc J.~Barzilai and J.~M. Borwein}, {\em Two-point step size gradient
            methods}, IMA Journal of Numerical Analysis, 8 (1988), pp.~141--148.
        
        \bibitem{bertsekas_99}
        {\sc D.~P. Bertsekas}, {\em Nonlinear {P}rogramming}, Athena Scientific,
        second~ed., 1999.
        
        \bibitem{byrd_1987}
        {\sc R.~H. Byrd, J.~Nocedal, and Y.-X. Yuan}, {\em Global convergence of a
            class of {Q}uasi-{N}ewton methods on convex problems}, SIAM Journal on
        Numerical Analysis, 24 (1987), pp.~1171--1190.
        
        \bibitem{conn_2000}
        {\sc A.~R. Conn, N.~I.~M. Gould, and P.~L. Toint}, {\em Trust-region methods},
        MPS-SIAM Series on Optimization, SIAM, 2000.
        
        \bibitem{crisci_2019}
        {\sc S.~Crisci, V.~Ruggiero, and L.~Zanni}, {\em Steplength selection in
            gradient projection methods for box-constrained quadratic programs}, Applied
        Mathematics and Computation, 356 (2019), pp.~312--327.
        
        \bibitem{dai_2003}
        {\sc Y.-H. Dai}, {\em Convergence properties of the {BFGS} algoritm}, SIAM
        Journal on Optimization, 13 (2002), pp.~693--701.
        
        \bibitem{deasmundis_2016}
        {\sc R.~De~Asmundis, D.~di~Serafino, and G.~Landi}, {\em On the regularizing
            behavior of the {SDA} and {SDC} gradient methods in the solution of linear
            ill-posed problems}, Journal of Computational and Applied Mathematics, 302
        (2016), pp.~81--93.
        
        \bibitem{diserafino_2020}
        {\sc D.~di~Serafino, G.~Landi, and M.~Viola}, {\em {ACQUIRE}: an inexact
            iteratively reweighted norm approach for {TV}-based {P}oisson image
            restoration}, Applied Mathematics and Computation, 364 (2020), p.~124678.
        
        \bibitem{diserafino_2018}
        {\sc D.~di~Serafino, V.~Ruggiero, G.~Toraldo, and L.~Zanni}, {\em On the
            steplength selection in gradient methods for unconstrained optimization},
        Applied Mathematics and Computation, 318 (2018), pp.~176--195.
        
        \bibitem{diserafino_2020lncs}
        {\sc D.~di~Serafino, G.~Toraldo, and M.~Viola}, {\em A gradient-based
            globalization strategy for the {N}ewton method}, in Numerical Computations:
        Theory and Algorithms. NUMTA 2019, Y.~D. Sergeyev and D.~E. Kvasov, eds.,
        vol.~11973 of Lecture Notes in Computer Science, Springer, 2020,
        pp.~177--185.
        
        \bibitem{diserafino_2018siopt}
        {\sc D.~di~Serafino, G.~Toraldo, M.~Viola, and J.~Barlow}, {\em A two-phase
            gradient method for quadratic programming problems with a single linear
            constraint and bounds on the variables}, SIAM Journal on Optimization, 28
        (2018), pp.~2809--2838.
        
        \bibitem{dolan_2002}
        {\sc E.~D. Dolan and J.~J. Mor\'e}, {\em Benchmarking optimization software
            with performance profiles}, Mathematical Programming, Series B, 91 (2002),
        pp.~201--213.
        
        \bibitem{dostal_18}
        {\sc Z.~Dost{\'a}l, G.~Toraldo, M.~Viola, and O.~Vlach}, {\em
            Proportionality-based gradient methods with applications in contact
            mechanics}, in High Performance Computing in Science and Engineering. HPCSE
        2017, T.~Kozubek, M.~{\v{C}}erm{\'a}k, P.~Tich{\'y}, R.~Blaheta,
        J.~{\v{S}}{\'\i}stek, D.~Luk{\'a}{\v{s}}, and J.~Jaro{\v{s}}, eds.,
        vol.~11087 of Lecture Notes in Computer Science, Springer, 2018, pp.~47--58.
        
        \bibitem{fang_08}
        {\sc H.~Fang and D.~P. O'Leary}, {\em Modified {C}holesky algorithms: a catalog
            with new approaches}, Mathematical Programming, 115 (2008), pp.~319--349.
        
        \bibitem{fletcher_2000}
        {\sc R.~Fletcher}, {\em Practical methods of optimization}, John Wiley \& Sons,
        second~ed., 2000.
        
        \bibitem{han_2003}
        {\sc L.~Han and M.~Neumann}, {\em Combining quasi-{N}ewton and {C}auchy
            directions}, International Journal of Applied Mathematics, 12 (2003),
        pp.~167--191.
        
        \bibitem{li_01}
        {\sc D.-H. Li and M.~Fukushima}, {\em A modified {BFGS} method and its global
            convergence in nonconvex minimization}, Journal of Computational and Applied
        Mathematics, 129 (2001), pp.~15--35.
        
        \bibitem{li_01a}
        \leavevmode\vrule height 2pt depth -1.6pt width 23pt, {\em On the global
            convergence of the {BFGS} method for nonconvex unconstrained optimization
            problems}, SIAM Journal on Optimization, 11 (2001), pp.~1054--1064.
        
        \bibitem{mascarenhas_04}
        {\sc W.~F. Mascarenhas}, {\em The {BFGS} method with exact line searches fails
            for non-convex objective functions}, Mathematical Programming, Series B, 99
        (2004), pp.~49--61.
        
        \bibitem{more_81}
        {\sc J.~J. Mor{\'e}, B.~S. Garbow, and K.~E. Hillstrom}, {\em Testing
            unconstrained optimization software}, ACM Transactions on Mathematical
        Software, 7 (1981), pp.~17--41.
        
        \bibitem{more_84}
        {\sc J.~J. Mor{\'e} and D.~C. Sorensen}, {\em Newton's method}, in Studies in
        {N}umerical {A}nalysis, G.~Golub, ed., The Mathematical Association of
        America, Providence, RI, 1984, pp.~29--82.
        
        \bibitem{murray_2011}
        {\sc W.~Murray}, {\em Newton-type methods}, in Wiley Encyclopedia of Operations
        Research and Management Science, Wiley, 2011.
        
        \bibitem{nesterov_04}
        {\sc Y.~Nesterov}, {\em Introductory lectures on convex optimization. A basic
            course}, vol.~87 of Applied Optimization, Springer Science+Business Media,
        2004.
        
        \bibitem{nocedal_06}
        {\sc J.~Nocedal and S.~J. Wright}, {\em Numerical Optimization}, Springer
        Series in Operations Research and Financial Engineering, Springer,
        second~ed., 2006.
        
        \bibitem{polak_97}
        {\sc E.~Polak}, {\em Optimization. Algorithms and consistent approximations},
        vol.~124 of Applied Mathematical Sciences, Springer, 1997.
        
        \bibitem{porta_2015}
        {\sc F.~Porta, M.~Prato, and L.~Zanni}, {\em A new steplength selection for
            scaled gradient methods with application to image deblurring}, Journal of
        Scientific Computing, 65 (2015), pp.~895--919.
        
        \bibitem{pospisil_2018}
        {\sc L.~Posp{\'\i}{\v s}il and Z.~Dost{\'a}l}, {\em The projected
            {B}arzilai--{B}orwein method with fall-back for strictly convex {QCQP}
            problems with separable constraints}, Mathematics and Computers in
        Simulation, 145 (2018), pp.~79--89.
        
        \bibitem{shi_1996}
        {\sc Y.~Shi}, {\em A globalization procedure for solving nonlinear systems of
            equations}, Numerical Algorithms, 12 (1996), pp.~273--286.
        
        \bibitem{shi_2000}
        \leavevmode\vrule height 2pt depth -1.6pt width 23pt, {\em Globally convergent
            algorithms for unconstrained optimization}, Computational Optimization and
        Applications, 16 (2000), pp.~295--308.
        
        \bibitem{zanella_2009}
        {\sc R.~Zanella, P.~Boccacci, L.~Zanni, and M.~Bertero}, {\em Efficient
            gradient projection methods for edge-preserving removal of {Poisson} noise},
        Inverse Problems, 25 (2009), p.~045010.
        
    \end{thebibliography}

\end{document}